\documentclass{mpaper}

\include{thref}
\usepackage{mathrsfs}
\usepackage{xargs}
\usepackage{ifthen}

\newcommand{\st}{{}\,:\,{}}
\newcommand{\textdef}[1]{\textit{#1}}

\newcommand{\restrict}[1]{\ensuremath{\left. \hspace{-1mm} \right|_{#1}}}

\newcommand{\diam}[1]{\mathrm{diam}(#1)}

\newcommand{\family}[2]{\{#1\}_{#2}}
\newcommand{\sequence}[2]{(#1)_{#2}}

\newcommandx{\cont}[3][1={},2={}]{{ \mathcal{C}^{#1}_{#2}\ifthenelse{\equal{#3}{}}{}{(#3)} }}

\newcommand{\setcomplement}[2][]{\ifthenelse{\equal{#1}{}}{#2^\mathrm{C}}{ {#1}\setdifference {#2} }}

\newcommand{\ProbabilityMeas}{\mathbb{P}}
\newcommand{\dd}[1]{\mathop{d#1}}
\newcommand{\support}{\operatorname{spt}}

\newcommand{\intring}[1]{\mathcal{O}(#1)}

\newcommand{\setdifference}{\setminus}

\newcommand{\boldfaceDelta}{\boldsymbol{\Delta}}
\newcommand{\boldfaceSigma}{\boldsymbol{\Sigma}}
\newcommand{\boldfacePi}{\boldsymbol{\Pi}}
\newcommand{\boldfaceGamma}{\boldsymbol{\Gamma}}

\newcommand{\borel}{\mathbf{B}}

\newcommand{\Cantor}{{2^\mathbb{N}}}

\newcommand{\hypClosed}{\mathbf{F}}
\newcommand{\hypCompact}{\mathbf{K}}
\newcommand{\hmetric}{\operatorname{d}_\mathcal{H}}

\newcommand{\hypCompactUF}{\mathbf{K}_{U}}

\newcommand{\hypClosedUF}{\mathbf{F}_U}

\newcommand{\hypClosedV}{\mathbf{V}}

\newcommand{\wadgereducible}{\le_W}

\newcommand{\Salem}{\mathscr{S}}
\newcommand{\closedSalem}{\Salem_c}

\newcommand{\hmeas}{\mathcal{H}}

\newcommand{\hdim}{ \operatorname{dim}_\mathcal{H} }
\newcommand{\capdim}{{\operatorname{dim}_c}}
\newcommand{\fourierdim}{\operatorname{dim}_{\mathrm{F}}}
\newcommand{\fourierdimcomp}{{\operatorname{dim}_{\mathrm{FC}}}}

\newcommand{\fouriertransform}[1]{{ \widehat{#1} }}
\newcommand{\scalarprod}[2]{{ {#1}\cdot {#2} }}

\newcommand{\closure}[1]{\overline{#1}}
\newcommand{\ball}[2]{{B(#1,#2)}}

\newcommand{\interior}{\operatorname{Int}}

\author{Alberto Marcone \and Manlio Valenti}
\date{}

\newcommand{\printauthor}{{% additional braces for segregating \footnotesize
  \bigskip
  \footnotesize

  Alberto Marcone, \textsc{Department of Mathematics, Computer Science and Physics\newline
  University of Udine\newline
  Udine, UD 33100, IT}\par\nopagebreak
  \textit{E-mail address}: \url{alberto.marcone@uniud.it}

  \medskip

  Manlio Valenti, \textsc{Department of Mathematics, Computer Science and Physics\newline
  University of Udine\newline
  Udine, UD 33100, IT}\par\nopagebreak
  \textit{E-mail address}: \url{manliovalenti@gmail.com}
}}

\def\subjclass[#1]#2{
	\begingroup
	\renewcommand\thefootnote{}\footnote{\textup{#1} \textit{Mathematics Subject Classification}: #2}%
	\addtocounter{footnote}{-1}%
	\endgroup
}

\date{}

\title{On the descriptive complexity of Salem sets}

\begin{document}
\maketitle

\begin{abstract}
	In this paper we study the notion of Salem set from the point of view of descriptive set theory. We first work in the hyperspace $\hypCompact([0,1])$ of compact subsets of $[0,1]$ and show that the closed Salem sets form a $\boldfacePi^0_3$-complete family. This is done by characterizing the complexity of the family of sets having sufficiently large Hausdorff or Fourier dimension. We also show that the complexity does not change if we increase the dimension of the ambient space and work in $\hypCompact([0,1]^d)$. 
	We then generalize the results by relaxing the compactness of the ambient space, and show that the closed Salem sets are still $\boldfacePi^0_3$-complete when we endow the hyperspace of all closed subsets of $\mathbb{R}^d$ with the Fell topology. A similar result holds also for the Vietoris topology.
\end{abstract}

\subjclass[2020]{Primary: 03E15; Secondary: 28A75, 28A78, 03D32.}

\tableofcontents

\section{Introduction}

The notion of Salem set arises naturally in the context of geometric measure theory and the theory of fractal dimension. A set $A\subset \mathbb{R}^d$ is called \textdef{Salem} iff $\hdim(A)=\fourierdim(A)$, where $\hdim$ and $\fourierdim$ denote the Hausdorff and the Fourier dimension respectively.

Hausdorff dimension is a fundamental notion in geometric measure theory and can be found in almost every textbook in the field. It describes the ``size'' of a set using the diameter of open sets covering it. When working with Borel subsets of $\mathbb{R}^d$, the Hausdorff dimension of a set can be characterized by the existence of finite Radon measures supported on the set with certain regularity properties (see Section~\ref{sec:backgroud} for details). 

This characterization establishes a close connection with the Fourier transform of a measure. Indeed, it can be shown that the decay of the Fourier transform of a (probability) measure supported on the set provides a lower bound for the Hausdorff dimension. This leads to the notion of Fourier dimension and hence to the one of Salem set. It is known that, for Borel subsets of $\mathbb{R}^d$, the Fourier dimension never exceeds the Hausdorff dimension. 

The first non-trivial examples of Salem sets were based on random constructions (\cite{Salem1950,Kahane1993}). Later Kahane \cite{Kahane1970} modified the original construction by Salem to produce an explicit Salem set of dimension $\alpha$, for every $\alpha\in [0,1]$. An important example of an explicit Salem set comes from the theory of Diophantine approximation of real numbers: Jarn\'ik \cite{Jarnik1928} and Besicovitch \cite{Besicovitch1934} proved that, for $\alpha \ge 0$, the set $E(\alpha)$ of $\alpha$-well approximable numbers is a fractal with Hausdorff dimension $2/(2+\alpha)$. Kaufmann \cite{Kaufmann81} improved the result by showing that there is a probability measure supported on a subset of $E(\alpha)$ witnessing the fact that $\fourierdim(E(\alpha))\ge 2/(2+\alpha)$, which implies that $E(\alpha)$ is Salem (the reader is referred to \cite{Bluhm98} or \cite{Wolff03} for detailed proofs of Kaufmann's theorem).

A classical example of a non-Salem set is Cantor middle-third set, which has Fourier dimension $0$ and Hausdorff dimension $\log(2)/\log(3)$. Similarly, every symmetric Cantor set with dissection ratio $1/n$, with $n>2$, is not a Salem set, as it has null Fourier dimension and Hausdorff dimension $\log(2)/\log(n)$ (see \cite[Sec.\ 4.10]{Mattila95} and \cite[Thm.\ 8.1]{MattilaFA}). It can be proved that, for every $0\le x \le y \le 1$ there is a compact subset of $[0,1]$ with Fourier dimension $x$ and Hausdorff dimension $y$ (\cite[Thm.\ 1.4]{Korner11}).

There are not many explicit (i.e.\ non-random) examples of subsets of $\mathbb{R}^d$ which are known to be Salem. As a corollary of a result of Gatesoupe \cite{Gatesoupe67}, we know that if $A\subset \mathbb{R}$ has at least two points and is a Salem set of dimension $\alpha$ then the set $\{ x \in \mathbb{R}^d \st |x|\in A \}$ is Salem and has dimension $d-1+\alpha$. Recently, using a higher-dimensional analogue of $E(\alpha)$, some explicit examples of Salem subsets of $\mathbb{R}^2$ and $\mathbb{R}^d$ (\cite{HambrookR2,FraserHambrook2019}) of arbitrary dimension have been constructed by Fraser and Hambrook.

In this paper we study the complexity, from the point of view of descriptive set theory, of the hyperspace of closed Salem subsets of $X$, where $X$ is either $[0,1]$, $[0,1]^d$ or $\mathbb{R}^d$. In other words, we study the complexity of the property ``being a Salem set'', when we restrict our attention to closed sets. We show that it is Borel and classify it in the Borel hierarchy. 

We summarize our results for $[0,1]$ in the following table.
\begin{center}
	\renewcommand{\arraystretch}{1.5}
	\begin{tabular}{ | c | c | c | }
		\hline 
		$p<1$ & $\{ A \in\hypCompact([0,1]) \st \hdim(A)> p\}$ & \hyperref[thm:dim>p_sigma02_complete]{$\boldfaceSigma^0_2$-complete} \\ 
		$p>0$ & $\{ A \in\hypCompact([0,1]) \st \hdim(A)\ge p\}$ & \hyperref[thm:dim_>=p_pi03_complete]{$\boldfacePi^0_3$-complete} \\ 
		$p<1$ & $\{ A \in\hypCompact([0,1]) \st \fourierdim(A)> p\}$& \hyperref[thm:dim>p_sigma02_complete]{$\boldfaceSigma^0_2$-complete} \\ 
		$p>0$ & $\{ A \in\hypCompact([0,1]) \st \fourierdim(A)\ge p\}$& \hyperref[thm:dim_>=p_pi03_complete]{$\boldfacePi^0_3$-complete} \\ 
		\hline
		\multicolumn{2}{| c |}{ $\{ A \in\hypCompact([0,1])\st A \text{ is Salem}\}$ }&\hyperref[thm:salem_pi03_complete]{$\boldfacePi^0_3$-complete} \\ 
		\hline
	\end{tabular}
\end{center}
\noindent The complexities remain the same if we replace $[0,1]$ with any interval, with $[0,1]^d$ or $\mathbb{R}^d$. 
In particular, the fact that the family of closed Salem subsets of $[0,1]$ is $\boldfacePi^0_3$-complete answers a question asked by Slaman during the IMS Graduate Summer School in Logic, held in Singapore in 2018.

\medskip

After reviewing the necessary background in Section \ref{sec:backgroud}, we devote Section \ref{sec:salem_[0,1]} to the case $[0,1]$. Here the proofs of the lower bounds for the complexities exploit Kaufmann's construction, and provide explicitly the details of the closed sets involved. In Section \ref{sec:salem_[0,1]d} we use a more abstract argument based on the Fraser/Hambrook's construction to extend the results to $[0,1]^d$. Section \ref{sec:R^d} is concerned with $\mathbb{R}^d$, and here more care is needed to prove the upper bounds.

Our results can be used to obtain the classifications of the functions computing the dimensions of closed sets, both in the Baire hierarchy and in the effective hierarchy defined via Weihrauch reducibility, in particular answering a question raised by Fouch\'e (\cite{Dagstuhl2016}) and Pauly. The details will be explored in a forthcoming paper.

\subsubsection*{Acknowledgements}

The early investigations leading to this paper were motivated by the above-mentioned question asked by Ted Slaman. The results of Section \ref{sec:salem_[0,1]} are indeed joint work with Ted Slaman and Jan Reimann. 

We would also like to thank Geoffrey Bentsen, Riccardo Camerlo, Kyle Hambrook, Arno Pauly, and Linda Brown Westrick for useful discussions and suggestions on the topics of the paper. We also thank the anonymous referee for his/her careful reading.

Both author's research was partially supported by the Italian PRIN 2017 Grant ``Mathematical Logic: models, sets, computability". Valenti's participation in the IMS Graduate Summer School in Logic in $2018$ was partially funded by the Institute for Mathematical Sciences (National University of Singapore).

\section{Background}
\label{sec:backgroud}
For a general introduction to geometric measure theory the reader is referred to \cite{Falconer14}. Here we introduce the notions and notations we will use throughout the rest of the paper. 

Let $X$ be a separable metric space and let $A\subset X$. Let also $\diam{A}$ denote the diameter of $A$. We say that a family $\family{E_i}{i\in I}$ is a $\delta$-cover of $A$ if $A\subset \bigcup_{i\in I} E_i$ and $\diam{E_i}\le \delta$ for each $i\in I$. For every $s\ge 0$, $\delta\in (0,+\infty]$ we define
\begin{align*}
	\hmeas^s_\delta(A) & := \inf\left\{ \sum_{i\in I} \diam{E_i}^s \st \family{E_i}{i\in I} \text{ is a }\delta\text{-cover of }A \right\},\\
	\hmeas^s(A) & := \lim_{\delta \to 0^+} \hmeas^s_\delta(A) = \sup_{\delta>0} \hmeas^s_\delta(A).
\end{align*}
The function $\hmeas^s$ is called \textdef{$s$-dimensional Hausdorff measure}. The \textdef{Hausdorff dimension} of $A$ is defined as 
\[ 	\hdim(A):= \sup\{ s\in [0,+\infty)\st \hmeas^s(A) >0\}. \]
As a consequence of Frostman's lemma (see \cite[Thm.\ 8.8, p.\ 112]{Mattila95}), for every Borel subset $A$ of $\mathbb{R}^d$ (with the Euclidean norm), the Hausdorff dimension of $A$ coincides with its capacitary dimension $\capdim(A)$, defined as 
\[ \sup \{ s\in [0,d] \st (\exists \mu \in \ProbabilityMeas(A))(\exists c>0)(\forall x \in \mathbb{R}^d)(\forall r>0) (\mu(\ball{x}{r})\le cr^s)\}, \]
where $\ProbabilityMeas(A)$ is the set of Borel probability measures with support included in $A$ and $\ball{x}{r}$ denotes the ball with center $x$ and radius $r$. We notice that the Hausdorff dimension is \textdef{countably stable} (i.e.\ for every family $\family{A_i}{i\in\mathbb{N}}$ we have $\hdim(\bigcup_i A_i) = \sup_i \hdim(A_i)$, see \cite[p.\ 59]{Mattila95}) and, for every $\alpha$-H\"older continuous map $f\colon\mathbb{R}^n\to \mathbb{R}^m$ we have $\hdim(f(A))\le \alpha^{-1}\hdim(A)$ (see \cite[Prop.\ 3.3, p.\ 49]{Falconer14}). In particular every bi-Lipschitz map preserves the Hausdorff dimension.

For every probability measure $\mu$ on $\mathbb{R}^d$, we can define the \textdef{Fourier transform of $\mu$} as the function $\fouriertransform{\mu}\colon\mathbb{R}^d\to \mathbb{C}$ defined as 
\[ \fouriertransform{\mu}(\xi) := \int_{\mathbb{R}^d} e^{-i\, \scalarprod{\xi}{x}} \dd{\mu(x)} \]
where $\scalarprod{\xi}{x}$ denotes the scalar product. We define the \textdef{Fourier dimension} of $A\subset \mathbb{R}^d$, written $\fourierdim(A)$, as 
\[ \sup\{ s\in [0,d]\st (\exists \mu\in \ProbabilityMeas(A))(\exists c>0)(\forall x\in \mathbb{R}^d)(|\fouriertransform{\mu}(x)|\le c|x|^{-s/2} ) \}. \]
If we define $\fourierdim(\mu):=\sup\{ s\in [0,d]\st (\exists c>0)(\forall x\in \mathbb{R}^d)\,(|\fouriertransform{\mu}(x)|\le c|x|^{-s/2} ) \}$ then we have $\fourierdim(A)=\sup\{ \fourierdim(\mu) \st \mu \in \ProbabilityMeas(A)\}$. For background notions on the Fourier transform the reader is referred to \cite{SteinWeiss}. For its applications to geometric measure theory see \cite{MattilaFA}.

The Fourier dimension is not as stable as the Hausdorff dimension. Some stability properties of the Fourier dimension have been investigated in \cite{EPS2014}. We underline, however, that the definition of Fourier dimension given in \cite{EPS2014} differs from the definition we use in this work (which agrees with the one that can be found in the literature \cite{Falconer14,Mattila95,MattilaFA, Wolff03}). The ``classical'' definition of Fourier dimension agrees with the \textdef{compact Fourier dimension} $\fourierdimcomp$ of \cite[Sec.\ 1.3]{EPS2014} (this can be showed, e.g., using \cite[Lem.\ 1]{EPS2014}). These notions agree if we restrict our attention to the dimension of closed sets. In general, requiring that the measure $\mu$ witnessing that $\fourierdim(A)>s$ gives full measure to $A$ is strictly weaker\footnote{In \cite[Ex.\ 7]{EPS2014}, the authors show that there is a set $X$ s.t.\ $X$ is a countable union of compact sets and $\fourierdimcomp(X)=\fourierdim(X)= 0$. However, admitting measures giving full measure to the set would give $X$ full dimension.} 
than requiring that $\mu$ is supported on $A$.

The fact that $\fourierdim=\fourierdimcomp$ implies that the Fourier dimension is inner regular for compact sets, i.e.\ 
\[ \fourierdim(A) = \sup \{ \fourierdim(K) \st K \subset A \text{ and } K \text{ is compact} \}. \]

On the other hand, the Fourier dimension is not finitely stable in general: the Bernstein set $B\subset \mathbb{R}$ (see \cite[Example 8.24]{KechrisCDST}) is s.t.\ every closed subset of $B$ or $\mathbb{R}\setdifference B$ is countable, and therefore $\fourierdim(B) = \fourierdim(\mathbb{R}\setdifference B) = 0$. On the other hand $\fourierdim(B \cup \mathbb{R}\setdifference B) = \fourierdim(\mathbb{R}) = 1$ (see also \cite[Sec.\ 1.3]{EPS2014}).

We can recover countable stability if we restrict our attention to closed sets: 

\begin{theorem}[{\cite[Prop.\ 5]{EPS2014}}]
	\thlabel{thm:fourier_sup_countable}
	If $\family{A_k}{k}$ is a finite or countable family of closed subsets of $\mathbb{R}^d$ then 
	\[ \fourierdim\left(\bigcup A_k\right) = \sup_k \fourierdim(A_k). \]
\end{theorem}

It is also known that the Fourier dimension does not behave well under H\"older continuous maps: there is a H\"older continuous transformation that maps the Cantor middle-third set to the interval $[0,1]$, although they have Fourier dimension respectively $0$ and $1$ (\cite[Sec.\ 8]{ESSurveyFourier}). However, the following fact, which we will use repeatedly in the paper, can be proved using the properties of the Fourier transform (see also \cite[Prop.\ 6]{ESSurveyFourier}):
\begin{fact}
	\thlabel{thm:fourier_invariant_affine}
	The Fourier dimension is invariant under affine invertible transformations.
\end{fact}

As a consequence of Frostman's lemma, for every Borel subset $A$ of $\mathbb{R}^d$, $\fourierdim(A)\le \hdim(A)$ (see \cite[Chap.\ 12]{Mattila95}). If $\fourierdim(A)= \hdim(A)$ then $A$ is called \textdef{Salem set}. When the Hausdorff and Fourier dimensions agree we drop the subscript and write only $\dim(\cdot)$. We denote with $\Salem(X)$ (resp.\ $\closedSalem(X)$) the collection of Salem subsets (resp.\ closed Salem subsets) of $X\subset \mathbb{R}^d$.

\medskip
In this work, we study the descriptive set-theoretic properties of the families $\closedSalem([0,1]^d)$ and $\closedSalem(\mathbb{R}^d)$. For an extended presentation of descriptive set theory the reader is referred to \cite{KechrisCDST}. 

Let $X$ be a metric space. It is known that the family of Borel subsets of $X$ can be stratified in a hierarchy, called the \textdef{Borel hierarchy}. Let $\omega_1$ be the first uncountable ordinal. The levels of the Borel hierarchy are defined by transfinite recursion on $1\le \xi< \omega_1$ as follows: we start from the families $\boldfaceSigma^0_1(X)$ and $\boldfacePi^0_1(X)$ of the open and the closed subsets of $X$ respectively. Then, for every $\xi>1$ we define:
\begin{itemize}
	\item[] $\boldfaceSigma^0_\xi(X) := \left\{\bigcup_n A_n \st A_n\in \boldfacePi^0_{\xi_n}(X),\, \xi_n< \xi,\, n\in\mathbb{N}\right\} $,
	\item[] $\boldfacePi^0_\xi(X) := \{ X\setdifference A \st A \in \boldfaceSigma^0_\xi(X) \} $.
\end{itemize}
Moreover, for every $\xi$, we define $\boldfaceDelta^0_\xi(X) := \boldfaceSigma^0_\xi(X) \cap \boldfacePi^0_\xi(X)$. In particular, $\boldfaceDelta^0_1(X)$ is the family of clopen subsets of $X$. The families $\boldfaceSigma^0_2(X)$ and $\boldfacePi^0_2(X)$ are often written resp.\ $\boldsymbol{F}_\sigma(X)$ and $\boldsymbol{G}_\delta(X)$. It is known that $\borel(X)=\bigcup_\xi \boldfaceSigma^0_\xi(X) = \bigcup_\xi \boldfacePi^0_\xi(X) = \bigcup_\xi \boldfaceDelta^0_\xi(X)$, where $\borel(X)$ denotes the family of Borel subsets of $X$. Whenever there is no ambiguity we will drop the dependency from the space $X$, and simply write $\boldfaceSigma^0_\xi$, $\boldfacePi^0_\xi$ and $\boldfaceDelta^0_\xi$.

Let $X$ and $Y$ be topological spaces and $A\subset X$, $B\subset Y$. We say that $A$ is \textdef{Wadge reducible} to $B$, and write $A \wadgereducible B$, if there is a continuous function $f\colon X\to Y$ s.t.\ $x\in A$ iff $f(x)\in B$. It is easy to see that if $\boldfaceGamma$ is among $\boldfaceSigma^0_\xi$, $\boldfacePi^0_\xi$, $\boldfaceDelta^0_\xi$ then $\boldfaceGamma$ is closed under continuous preimages, i.e.\ $A\wadgereducible B$ and $B\in \boldfaceGamma(Y)$ implies $A \in \boldfaceGamma(X)$.

Fix a class $\boldfaceGamma$ as above. Assume $Y$ is a Polish (i.e.\ separable and completely metrizable) space and $B\subset Y$.  We say that $B$ is \textdef{$\boldfaceGamma$-hard} if $A\wadgereducible B$ for every $A\in \boldfaceGamma(X)$, where $X$ is Polish and has a basis consisting of clopen sets. If $B$ is $\boldfaceGamma$-hard and $B \in \boldfaceGamma(Y)$ then we say that $B$ is \textdef{$\boldfaceGamma$-complete}.

A common technique to show that a set $B\subset X$ is $\boldfaceGamma$-hard is to show that there is a Wadge reduction $A\wadgereducible B$, for some $A$ which is already known to be $\boldfaceGamma$-complete. Standard examples of $\boldfaceGamma$-complete sets are the following (see \cite[Sec.\ 23.A, p.\ 179]{KechrisCDST}):
\begin{center}
	\renewcommand{\arraystretch}{1.5}
	\begin{tabular}{  l l  }
		$Q_2:= \{ x\in \Cantor \st (\forall^\infty m)(x(m)=0)\}$ & $\boldfaceSigma^0_2$-complete, \\ 
		$N_2:= \{ x\in \Cantor \st (\exists^\infty m)(x(m)=0)\}$ & $\boldfacePi^0_2$-complete, \\ 
		$S_3:= \{ x\in 2^{\mathbb{N}\times \mathbb{N}} \st (\exists k)(\exists^\infty m)(x(k,m)=0)\}$ & $\boldfaceSigma^0_3$-complete, \\ 
		$P_3:= \{ x\in 2^{\mathbb{N}\times \mathbb{N}} \st (\forall k)(\forall^\infty m)(x(k,m)=0)\}$ & $\boldfacePi^0_3$-complete,
	\end{tabular}
\end{center}
where $(\exists^\infty m)$ and $(\forall^\infty m)$ mean respectively $(\forall n)(\exists m\ge n)$ and $(\exists n)(\forall m\ge n)$.

For a topological space $X$, we denote by $\hypClosed(X)$ and $\hypCompact(X)$ respectively the hyperspaces of closed and compact subsets of $X$. 

There is no canonical choice for the topology on $\hypClosed(X)$, and several alternatives have been explored in the literature \cite{Beer1993,KleinThom84}. Let $\mathscr{U}$ be the collection of sets of the form
\[ \{F\in \hypClosed(X) \st F \cap C = \emptyset \},\]
where $C$ ranges over all closed subsets of $X$. The topology having $\mathscr{U}$ as a prebase is called \textdef{upper topology} or \textdef{upper Vietoris topology} (\cite[Def.\ 1.3.1]{KleinThom84}). In the same spirit, we can define $\mathscr{L}$ as the family of sets of the form
\[\{F\in \hypClosed(X) \st F \cap U \neq \emptyset\},\]
where $U$ ranges over the open subsets of $X$. The topology having $\mathscr{L}$ as a prebase is called \textdef{lower topology} or \textdef{lower Vietoris topology} (\cite[Def.\ 1.3.2]{KleinThom84}). The \textdef{Vietoris topology} is the topology having as a prebase the family $\mathscr{L} \cup \mathscr{U}$.

The Vietoris topology is not always the preferred choice. As an alternative, we can consider the collection $\mathscr{U}_\mathcal{K}$ of sets of the form
\[ \{F\in \hypClosed(X) \st F \cap K = \emptyset \},\]
where $K$ ranges over all compact subsets of $X$. The family $\mathscr{U}_\mathcal{K}$ is a prebase for a topology on $\hypClosed(X)$ called \textdef{upper Fell topology}. We can define the \textdef{Fell topology} on $\hypClosed(X)$ as the topology having as a prebase the set $\mathscr{U}_\mathcal{K} \cup \mathscr{L}$. For this reason, the lower Vietoris topology is often called \textdef{lower Fell topology}. 

Unlike the hyperspace $\hypClosed(X)$, there is a canonical choice for the topology for the hyperspace $\hypCompact(X)$ of compact subsets of $X$. In fact, $\hypCompact(X)$ is usually endowed with the topology induced from the Vietoris topology on $\hypClosed(X)$.

If $X$ is a bounded metric space with distance $d$, we can define the \textdef{Hausdorff metric} $\hmetric$ on $\hypCompact(X)$ as follows:
\[ \hmetric(K,L) := \begin{cases}
	0 & \text{if } K=L=\emptyset\\
	\diam{X} & \text{if exactly one among }K \text{ and } L\text{ is }\emptyset\\
	\max\{\delta(K,L),\delta(L,K) \} &\text{otherwise,}
\end{cases} \]
where $\delta(K,L) := \max_{x\in K} d(x,L)$. It is known that the Hausdorff metric $\hmetric$ is compatible with the Vietoris topology on $\hypCompact(X)$ (\cite[Ex.\ 4.21]{KechrisCDST}) and that if $X$ is Polish then so is $\hypCompact(X)$ (\cite[Thm.\ 4.22]{KechrisCDST}). 

The choice of the Vietoris topology is, of course, not the only possible: any topology on $\hypClosed(X)$ induces a topology on $\hypCompact(X)$. In particular, we will write $\hypCompactUF(X)$ for the hyperspace of compact subsets of $X$ endowed with the upper Fell topology.

One of the main reasons why the Vietoris topology is not the canonical choice for $\hypClosed(X)$ is that it is not paracompact, and hence metrizable\footnote{Intuitively, the $\max$ in the definition of $\delta(K,L)$ is not guaranteed to exist, and two closed sets can be infinitely distant.}, if $X$ is not compact (\cite[Thm.\ 2]{Ke70}). On the other hand, if $X$ Polish and locally compact then the Fell topology on $\hypClosed(X)$ gives rise to a Polish and its Borel space is exactly the Effros-Borel space (\cite[Ex.\ 12.7]{KechrisCDST}). Notice that the Fell and the Vietoris topologies coincide if $X$ is compact.

An important topological space is the space of Borel probability measures. If $X$ is a separable metrizable space, we consider the space $\ProbabilityMeas(X)$ of Borel probability measures on $X$, endowed with the topology generated by the maps $\mu \mapsto \int f\dd\mu$, with $f\in\cont[][b]{X}$ (i.e.\ $f\colon X \to \mathbb{R}$ is continuous and bounded, see \cite[Sec.\ 17.E, p.\ 109]{KechrisCDST}). A basis for the topology on $\ProbabilityMeas(X)$ is the family of sets of the form 
\[ U_{\mu,\varepsilon,f_0,\hdots,f_n} := \left\{ \nu\in\ProbabilityMeas(X) \st (\forall i\le n)\left( \left| \int_X f_i \dd{\nu} - \int_X f_i \dd{\mu} \right| < \varepsilon\right) \right\}, \]
where $\mu\in\ProbabilityMeas(X)$, $\varepsilon>0$, and $f_i\in\cont[][b]{X}$  for every $i$. If $X$ is compact metrizable then so is $\ProbabilityMeas(X)$ (\cite[Thm.\ 17.22]{KechrisCDST}). Moreover, if $X$ is Polish then so is $\ProbabilityMeas(X)$ (\cite[Thm.\ 17.23]{KechrisCDST}).

We conclude this section with the following lemma:
\begin{lemma}[{\cite[Lem.\ 1.3]{AndrMarcODE97}}]
	\thlabel{thm:AM_Sigma02}
	Let $X$ be Polish and $Y$ metrizable and $\mathbf{K}_\sigma$ (i.e.\ countable union of compact sets). If $F\subset X\times Y$ is $\boldfaceSigma^0_2$ then $\operatorname{proj}_X(F)$ is also $\boldfaceSigma^0_2$.
\end{lemma}

\section{The complexity of closed Salem subsets of \texorpdfstring{$[0,1]$}{[0,1]} }
\label{sec:salem_[0,1]}

In this section, we determine the exact complexity of the family of closed Salem subsets of $[0,1]$. We first obtain an upper bound for the complexity of the conditions $\hdim(A)>p$, $\hdim(A)\ge p$, $\fourierdim(A)>p$ and $\fourierdim(A)\ge p$. Since the upper Fell topology is coarser than the Vietoris topology, % the inclusion $\hypCompact([0,1])\hookrightarrow \hypCompactUF([0,1])$ is continuous. This implies that 
obtaining an upper bound for the above conditions when the hyperspace of compact subsets of $[0,1]$ is endowed with the upper Fell topology immediately yields an upper bound for the same conditions when the hyperspace is endowed with the Vietoris topology instead.

\begin{lemma}
	\thlabel{thm:mu(K)>=x_closed}
	Let $X$ be a closed subset of $\mathbb{R}^d$. The set 
	\[ \{ (\mu,K,x) \in \ProbabilityMeas(X)\times \hypCompactUF(X) \times \mathbb{R} \st \mu(K) \ge x \}\]
	is closed.
\end{lemma}
\begin{proof}
	We prove that the complement is open. Let $(\mu,K,x)$ be s.t.\ $\mu(K)=x-\varepsilon<x$. By the outer regularity of $\mu$, there are two open sets $U,V$ s.t.\ 
	\begin{itemize}
		\item $K\subset U \subset \closure{U} \subset V \subset X$,
		\item $\mu(V)<x-\varepsilon/2$.
	\end{itemize}
	Similarly, by the inner regularity of $\mu$, there are two open sets $W,Z$ s.t.\ 
	$X\setdifference W$ is compact and 
	\begin{itemize}
		\item $W\subset \closure{W} \subset Z$,
		\item $\mu(Z)<\varepsilon/8$.
	\end{itemize}

	By Urysohn's lemma, there are two continuous functions $f,g\colon X \to [0,1]$ s.t.\ $f(\closure{U})=1=g(\closure{W})$ and $f(X\setdifference V)=0=g(X \setdifference Z)$.
	
	Recall that the set $U_{\mu,\varepsilon/16,f,g}$, i.e.\
	\[ \left\{ \nu\in \ProbabilityMeas(X)\st \left| \int f \dd{\nu} - \int f \dd{\mu} \right| <\frac{\varepsilon}{16} \land \left| \int g \dd{\nu} - \int g \dd{\mu} \right| <\frac{\varepsilon}{16}\right\} \]
	is a basic open set for $\ProbabilityMeas(X)$. Define the set 
	\[		\mathcal{U}:=\{ H \in \hypCompactUF(X) \st H \subset U\cup W \}.	\]
	Notice that $U \cup W$ has compact complement, hence $\mathcal{U}$ is a basic open subset of $\hypCompactUF(X)$.

	We claim that for every $(\nu, H, y)\in U_{\mu,\varepsilon/16,f,g}\times \mathcal{U}\times \ball{x}{\varepsilon/4}$ we have $\nu(H)<y$. Indeed 
	\begin{align*}
		\nu(H) & \le \nu(U)+\nu(W) \le \int f \dd{\nu} + \int g \dd{\nu} \\
			& \le \int f \dd{\mu}+\int g \dd{\mu}+\frac{\varepsilon}{8}\\
			& \le \mu(V) + \mu(Z) +\frac{\varepsilon}{8} < x - \frac{\varepsilon}{4}<y.  \qedhere
	\end{align*}
\end{proof}

Notice that the same set is not closed if we consider the lower Fell topology on $\hypCompact(X)$, essentially because $X$ belongs to every non-empty open set of it.
 
\begin{proposition}
	\thlabel{thm:capdim_complexity}
	${}$
	\begin{itemize}
		\item \label{itm:dim_c>p_is_S02} $\{ (A,p) \in\hypCompactUF([0,1])\times [0,1] \st \hdim(A)> p\}$ is $\boldfaceSigma^0_2$;
		\item \label{itm:dim_c>=p_is_P03} $\{ (A,p) \in\hypCompactUF([0,1])\times [0,1] \st \hdim(A)\ge p\}$ is $\boldfacePi^0_3$.
	\end{itemize}
\end{proposition}
\begin{proof}
	As noticed in the previous section, for Borel (in particular closed) $A \subset [0,1]$, the Hausdorff dimension $\hdim(A)$ coincides with the capacitary dimension $\capdim(A)$. For ease of readability, we denote with $D(A)$ the set  
\[ \{ s\in [0,1] \st (\exists \mu \in \ProbabilityMeas(A))(\exists c>0)(\forall x\in \mathbb{R})(\forall r>0)(\mu(\ball{x}{r})\le c r^s ) \}. \]	
In particular $\hdim(A)=\capdim(A) = \sup D(A)$. Notice that $D(A)$ is downward closed. Clearly
\[ \mu(\ball{x}{r})\le c r^s \iff \mu(\setcomplement[{[0,1]}]{\ball{x}{r}})\ge 1-c r^s. \]
Observe that the map $(x,r)\mapsto [0,1]\setdifference \ball{x}{r}$ is continuous when the codomain is endowed with the Vietoris topology. In particular, it is continuous as a function $\mathbb{R}^2 \to \hypCompactUF([0,1])$.  By \thref{thm:mu(K)>=x_closed} the condition $\mu(\ball{x}{r})\le c r^s$ is closed, hence the set 
\[ C:= \{ (s,c,\mu)\st (\forall x\in \mathbb{R})(\forall r>0)(\mu(\ball{x}{r})\le c r^s ) \} \]
is a closed subset of the product space $[0,1]\times [0,+\infty) \times \ProbabilityMeas(A)$. Notice also that 
\[ \mu\in \ProbabilityMeas(A)\iff \mu\in\ProbabilityMeas([0,1]) \text{ and } \mu(A)\ge 1. \]
Since the condition $\mu(A)\ge 1$ is closed (again, by \thref{thm:mu(K)>=x_closed}) we have that, for each closed subset $A$ of $[0,1]$, the set
\[ Q:= \{ (s,\mu)\in [0,1]\times \ProbabilityMeas([0,1])\st (\exists c>0) (\mu\in \ProbabilityMeas(A) \land (s,c,\mu) \in C )\} \]
is $\boldfaceSigma^0_2$. 

Recall that the space $\ProbabilityMeas([0,1])$ is metrizable and compact. Using \thref{thm:AM_Sigma02} we can conclude that the set $D(A) = \operatorname{proj}_{[0,1]} Q$ is $\boldfaceSigma^0_2$. To conclude the proof we notice that the conditions
\[ \hdim(A) > p \iff (\exists s\in \mathbb{Q})(s>p \land s\in D(A)),\]
\[ \hdim(A) \ge p \iff (\forall s\in \mathbb{Q})(s<p \rightarrow s\in D(A))\]
are $\boldfaceSigma^0_2$ and $\boldfacePi^0_3$ respectively.
\end{proof}

\begin{proposition}
	\thlabel{thm:fourierdim_complexity}
	${}$
	\begin{itemize}
		\item \label{itm:dim_F>p_is_S02} $\{ (A,p) \in\hypCompactUF([0,1])\times [0,1] \st \fourierdim(A)> p\}$ is $\boldfaceSigma^0_2$;
		\item \label{itm:dim_F>=p_is_P03} $\{ (A,p) \in\hypCompactUF([0,1])\times [0,1] \st \fourierdim(A)\ge p\}$ is $\boldfacePi^0_3$.
	\end{itemize}
\end{proposition}
\begin{proof}
	For the sake of readability, let 
	\[ D(A):=\{s\in[0,1]\st (\exists \mu \in \ProbabilityMeas(A))(\exists c>0)(\forall x\in \mathbb{R})(|\fouriertransform{\mu}(x)|\le c|x|^{-s/2}) \}.\] 
	First of all we notice that the condition $|\fouriertransform{\mu}(x)|>c |x|^{-s/2}$ is $\boldfaceSigma^0_1$. To see this it is enough to show that the map $F\colon \ProbabilityMeas([0,1])\times \mathbb{R}\to \mathbb{R}$ s.t.\ $F(\mu,x)=|\fouriertransform{\mu}(x)|$ is continuous. Indeed, if that is the case, then the tuple $(\mu,x,s,c)$ satisfies the condition $|\fouriertransform{\mu}(x)|>c |x|^{-s/2}$ iff it belongs to the preimage of $(0,+\infty)$ via the map $(\mu,x,s,c)\mapsto F(\mu,x)-c|x|^{-s/2}$, which is clearly continuous.
	 
	Recall that, for each finite Borel measure $\mu$, the Fourier transform $\fouriertransform{\mu}$ is a bounded uniformly continuous function. 
	
	Notice that the set 
	\[ V_{\mu,\varepsilon,x}:=\{ \nu \in \ProbabilityMeas([0,1])\st |\fouriertransform{\mu}(x)-\fouriertransform{\nu}(x)|<\varepsilon \} \]
	is open in the topology of $\ProbabilityMeas([0,1])$. Indeed, fix $\nu \in V_{\mu,\varepsilon,x}$ and let $\delta$ s.t.\ $|\fouriertransform{\mu}(x) -\fouriertransform{\nu}(x)|+\delta < \epsilon$. %Using the notation of \cite[p.\ 110]{KechrisCDST}, 
	We claim that the basic open set $U_{\nu,\frac{\delta}{2},\cos(x\,\cdot), \sin(x\,\cdot) }$ is included in $V_{\mu,\varepsilon,x}$.
	In fact, for each $\eta \in U_{\nu,\frac{\delta}{2},\cos(x\,\cdot), \sin(x\,\cdot) }$ we have 
	\begin{align*}
		|\fouriertransform{\nu} & (x) -\fouriertransform{\eta}(x)|  = \left| \int e^{-i\,xt}\dd{\nu}(t)-\int e^{-i\,xt}\dd{\eta}(t)\right| \le \\
			& \le \left| \int \cos(x t)\dd{\nu(t)} - \int \cos(x t)\dd{\eta(t)}\right| + \left| \int \sin(x t)\dd{\nu(t)} - \int \sin( x t)\dd{\eta(t)} \right|\le \delta
	\end{align*}
	and therefore
	\[ |\fouriertransform{\mu}(x) -\fouriertransform{\eta}(x)| \le |\fouriertransform{\mu}(x) -\fouriertransform{\nu}(x)| + |\fouriertransform{\nu}(x) -\fouriertransform{\eta}(x)| \le |\fouriertransform{\mu}(x) -\fouriertransform{\nu}(x)|+\delta<\varepsilon.\]

	To conclude the proof of the continuity	we show that for each $\varepsilon>0$ and each $\mu$ and $x$ we can choose $\delta$ sufficiently small s.t.\ for every $(\nu,y)\in V_{\mu,\delta,x} \times \ball{x}{\delta}$ we have $|\fouriertransform{\mu}(x)-\fouriertransform{\nu}(y)| < \varepsilon$. Indeed, by the triangle inequality
	\[ |\fouriertransform{\mu}(x)-\fouriertransform{\nu}(y)| \le |\fouriertransform{\mu}(x)-\fouriertransform{\nu}(x)| + |\fouriertransform{\nu}(x)-\fouriertransform{\nu}(y)|.  \]
	The first term is bounded by $\delta$ by definition of $V_{\mu,\delta,x}$. Moreover
	\begin{align*}
		|\fouriertransform{\nu}(x)&-\fouriertransform{\nu}(y)|  = \left| \int e^{-ixt}-e^{-iyt} \dd{\nu(t)} \right| \le \\
			& \le \int |\cos(xt)-\cos(yt)| \dd{\nu(t)}+  \int |\sin(xt)-\sin(yt)| \dd{\nu(t)}.
	\end{align*}
	By the sum-to-product formulas
	\begin{align*}
		\int |\cos(xt)-\cos(yt)| \dd{\nu(t)} & = \int 2\left| \sin\left(\frac{(x+y)t}{2}\right)\sin\left(\frac{(x-y)t}{2}\right) \right| \dd{\nu(t)} \le \\
			& \le 2 \sin\left(\frac{x-y}{2}\right) 
	\end{align*}
	and similarly
	\[\int |\sin(xt)-\sin(yt)| \dd{\nu(t)} \le  2 \sin\left(\frac{x-y}{2}\right). \]
	hence the claim follows.
	
	Since $\mu\in\ProbabilityMeas(A)$ is a closed condition (see the proof of \thref{thm:capdim_complexity}), the set 
	\begin{align*}
		\{ (s,c,\mu)\in [0,1]\times [0,+\infty) \times\ProbabilityMeas([0,1]) \st & \mu\in\ProbabilityMeas(A) \land{}\\
			& (\forall x\in \mathbb{R}^d) (|\fouriertransform{\mu}(x)|\le c |x|^{-s/2}) \} 
	\end{align*}
	is closed and, therefore, the set 
	\begin{align*}
		Q:= \{ (s,\mu)\in [0,1]\times\ProbabilityMeas([0,1]) \st (\exists c>0)(\forall x \in \mathbb{R}^d)(& \mu\in\ProbabilityMeas(A) \land{}\\
			&  |\fouriertransform{\mu}(x)|\le c|x|^{-s/2} ) \}
	\end{align*}
	is $\boldfaceSigma^0_2$. As in the proof of \thref{thm:capdim_complexity}, we can use \thref{thm:AM_Sigma02} to conclude that the set $D(A) = \operatorname{proj}_{[0,1]} Q$ is $\boldfaceSigma^0_2$ and hence the conditions 	
	\[ \fourierdim(A) > p \iff (\exists s\in \mathbb{Q})(s>p \land s\in D(A)),\]
	\[ \fourierdim(A) \ge p \iff (\forall s\in \mathbb{Q})(s<p \rightarrow s\in D(A))\]
	are $\boldfaceSigma^0_2$ and $\boldfacePi^0_3$ respectively.
\end{proof}
	
\begin{theorem}
	\thlabel{thm:salem_pi03}
	The set $\{ A \in\hypCompactUF([0,1])\st A \in \Salem([0,1])\}$	is $\boldfacePi^0_3$. 
	% \begin{itemize}
	% 	\item $$ is $\boldfacePi^0_3$,
	% 	\item $\{ (A,p) \in\closedSalem([0,1])\times [0,1] \st \dim(A)> p\}$ is $\boldfaceSigma^0_2$,
	% 	\item $\{ (A,p) \in\closedSalem([0,1])\times [0,1] \st \dim(A)\ge p\}$ is $\boldfacePi^0_3$.
	% \end{itemize}
\end{theorem}
\begin{proof}
	To prove that $ A \in \Salem([0,1])$ is a $\boldfacePi^0_3$ condition recall that, for Borel subsets of $\mathbb{R}^d$, $\fourierdim(A)\le \hdim(A)$.
	For a closed subset $A$ of $[0,1]$, the condition $\hdim(A)=\fourierdim(A)$ can be written as 
	\[ (\forall r\in \mathbb{Q})( \hdim(A)>r \rightarrow \fourierdim(A)>r ). \]
	The claim follows from \thref{thm:capdim_complexity} and \thref{thm:fourierdim_complexity}, as both the conditions $\hdim(A)>r$ and $\fourierdim(A)>r$ are $\boldfaceSigma^0_2$.
\end{proof}

We now show that the above conditions are complete for their respective classes (i.e.\ the upper bounds are tight) when the hyperspace of compact subsets of $[0,1]$ is endowed with the Vietoris topology. Since the Vietoris topology is finer than the upper Fell topology, the same lower bounds hold when the hyperspace of compact subsets of $[0,1]$ is endowed with the upper Fell topology.

The proof of the following \thref{thm:sigma02_map_salem} exploits the properties of the set $E(\alpha)$ of $\alpha$-well approximable numbers.

\begin{definition}[{\cite[Sec.\ 10.3, p.\ 172]{Falconer14}}]
	For every $\alpha\ge 0$, we say that $x\in [0,1]$ is \textdef{$\alpha$-well approximable} if there are infinitely many $n\in \mathbb{N}$ s.t.\ 
	\[ \min_{m\in\mathbb{Z}} | nx - m | \le n^{-1-\alpha}. \]
	The set of $\alpha$-well approximable numbers is denoted by $E(\alpha)$.
\end{definition}

As mentioned in the introduction, $E(\alpha)$ is a Salem set of dimension $2/(2+\alpha)$. Notice that, by definition, the set $E(\alpha)$ is $\boldfacePi^0_3$, as it can be written in the form
\[ E(\alpha) = \bigcap_{k\in\mathbb{N}} \bigcup_{n \ge k} G_n,\]
where $G_n:=\{ x\in [0,1] \st \min_{m\in\mathbb{Z}} | nx - m | \le n^{-1-\alpha} \}$ is a closed set (it is a finite union of non-degenerate closed intervals).

If $\alpha=0$ then, by Dirichlet's theorem (\cite[Ex.\ 10.8]{Falconer14}), $E(\alpha)=[0,1]$. However, if $\alpha>0$ then $E(\alpha)$ is not closed (because $E(\alpha)$ is dense in $[0,1]$ but does not have full dimension). 

In the construction presented in \cite{Bluhm98}, the author explicitly writes the support\footnote{In \cite{Bluhm98} it is denoted with $S_\alpha$.} $S(\alpha)$ of a measure witnessing that $\fourierdim(E(\alpha))\ge 2/(2+\alpha)$. This, in particular, implies that $S(\alpha)$ itself is Salem with dimension $2/(2+\alpha)$. The set $S(\alpha)$ can be written as 
\[ S(\alpha) = \bigcap_{k\in\mathbb{N}} \bigcup_{k'\le n \le k''} G_n. \]
In other words, it is obtained from $E(\alpha)$ by making the inner union finite, where $k'$ and $k''$ depend on $k$ and are strictly increasing. Clearly $S(\alpha)$ is closed (as it is the infinite intersection of closed sets). We can rewrite $S(\alpha)$ as follows:
\[ S(\alpha) = \bigcap_{k\in\mathbb{N}} S^{(k)}(\alpha) \]
where 
\[ S^{(k)}(\alpha) = \bigcup_{i\le M_k} I_i(\alpha,k) \]
and, for each $k$, the $I_i(\alpha,k)$ are disjoint non-degenerate closed intervals. 

% We consider a slight variation $R(\alpha)$ of $S(\alpha)$ s.t.\ if $h<k$ then $R^{(k)}(\alpha) \subset R^{(h)}(\alpha)$, and, moreover, for every $i\le M_h$ there exists $j\le M_k$ s.t.\ $I_j(\alpha,k)\subset I_i(\alpha,h)$.
We modify $S(\alpha)$ to obtain 
\[ R(\alpha) = \bigcap_{k\in\mathbb{N}} R^{(k)}(\alpha) =  \bigcap_{k\in\mathbb{N}} \bigcup_{j\le N_k} J_j(\alpha,k), \]
where each $J_j(\alpha,k)$ is a non-degenerate closed interval, with the property that $R^{(k+1)}(\alpha) \subset R^{(k)}(\alpha)$, and, moreover, for every $i\le N_k$ there exists $j\le N_{k+1}$ s.t.\ $J_j(\alpha,k+1)\subset J_i(\alpha,k)$. To this end, define $R^{(k)}(\alpha)$ inductively as follows: $R^{(0)}(\alpha):=S^{(0)}(\alpha)$. At stage $k+1$, let  %let $R^{(k)}(\alpha)$ s.t.\ $S^{(k)}(\alpha) \setdifference R^{(k)}(\alpha)$ is finite. Let  
\[ \tilde R^{(k+1)}(\alpha):= S^{(k+1)}(\alpha) \cup \bigcup_{n \in U_k} G_n, \] 
where $U_k\subset \mathbb{N}$ is a finite set of indexes s.t.\ for every interval $j\le N_k$,
\[\interior(J_j(\alpha,k)) \cap \tilde R^{(k+1)}(\alpha) \neq \emptyset,\]
where $\interior(\cdot)$ denotes the interior. Such a choice of $U_k$ is always possible by the density of $E(\alpha)$. We obtain $R^{(k+1)}(\alpha)$ by considering the finitely many intervals whose union is $\tilde R^{(k+1)}(\alpha)\cap R^{(k)}(\alpha)$ and removing the degenerate ones. 

Notice that, for every $k$, $S^{(k)}(\alpha)\setminus R^{(k)}(\alpha)$ is finite. This implies that $S(\alpha) \setminus R(\alpha)$ is countable and therefore, by \thref{thm:fourier_sup_countable}, $\fourierdim(S(\alpha))=\fourierdim(R(\alpha))$. Notice, moreover, that $R(\alpha) \subset E(\alpha)$, and therefore $R$ is still a Salem set and $\dim(R(\alpha))=2/(2+\alpha)$.

\begin{lemma}
	\thlabel{thm:sigma02_map_salem} 
	For every $p\in [0,1]$ there exists a continuous map $f_p\colon \Cantor \to \hypCompact([0,1])$ s.t.\ for every $x$, $f_p(x)$ is Salem and  
	\[ \dim(f_p(x)) = \begin{cases}
		p & \text{if } x\in Q_2\\
		0 & \text{if } x\notin Q_2
	\end{cases}\]
\end{lemma}
\begin{proof}
	Recall that $Q_2 = \{ x\in \Cantor \st (\forall^\infty k) (x(k)=0) \}$
	is $\boldfaceSigma^0_2$-complete. %(see \cite[p.\ 179]{KechrisCDST}).

	The case $p=0$ is trivial (just take the constant map $x\mapsto \emptyset$), so assume $p>0$. Let $\alpha\ge 0$ s.t.\ $2/(2+\alpha)=p$ and consider the Salem set $S(\alpha)$ as defined above.

	For each $x\in\Cantor$ we define a sequence $\sequence{F_x^{(k)}}{k\in\mathbb{N}}$ of nested closed sets s.t.\ each $F_x^{(k)}$ is a finite union of closed intervals. The idea is to follow the construction of $R(\alpha)$ until we find a $k$ s.t.\ $x(k)=1$. If this never happens then in the limit we obtain $R(\alpha)$, which is a Salem set of Fourier dimension $p$. On the other hand, each time we find a $k$ s.t.\ $x(k)=1$ we modify the next step of the construction by replacing each of the (finitely many) intervals $J_0,\hdots,J_{N_k}$ whose union is the $k$-th level of the construction with sufficiently small subintervals $H_0,\hdots,H_{N_k}$, and we reset the construction, starting again a (proportionally scaled down) construction of $R(\alpha)$ on each subinterval $H_i$. By carefully choosing the length of the subintervals $H_i$ we can ensure that, if there are infinitely many $k$ s.t.\ $x(k)=1$ then $F_x$ has null Hausdorff (and hence Fourier) dimension. 

	Formally, if $I=[a,b]$ is an interval then we define $R(\alpha, I)$ as the fractal obtained by scaling $R(\alpha)$ to the interval $I$. Notice that, by \thref{thm:fourier_invariant_affine}, $R(\alpha,I)$ is still a Salem set of dimension $p$. 
	
	We define $F^{(k)}_x$ recursively as
	\begin{description}
		\item[Stage $k=0$]: $F^{(0)}_x := [0,1]$;
		\item[Stage $k+1$]: Let $J_0,\hdots,J_{N_k}$ be the disjoint closed intervals s.t.\ $F^{(k)}_x=\bigcup_{i\le N_k} J_i$. If $x(k+1)=1$ then choose, for each $i\le N_k$, a (non-degenerate) subinterval $H_i=[a_i,b_i] \subset J_i$ so that  
		\[ \sum_{i\le N_k} \diam{H_i}^{2^{-k}} \le 2^{-k}. \]
		Define then $F^{(k+1)}_x:= \bigcup_{i\le N_k} H_i$.
		
		If $x(k+1)=0$ then let $s\le k$ be largest s.t.\ $x(s)=1$ (or $s=0$ if there is none) and let $I_0,\hdots,I_{N_s}$ be the intervals of $F^{(s)}_x$. For each $i\le N_s$, apply the $(k+1-s)$-th step of the construction of $R(\alpha, I_i)$. Define $F^{(k+1)}_x:=\bigcup_{i\le N_s} R^{(k+1-s)}(\alpha, I_i)$.
	 \end{description}  
	 
	We define the map $f_p$ as $f_p(x):= F_x=\bigcap_{k\in\mathbb{N}} F^{(k)}_x$. Clearly $F_x$ is closed, as intersection of closed sets. To show that $f_p$ is continuous, recall that the Vietoris topology is compatible with the Hausdorff metric $\hmetric$.
	Fix $x\in \Cantor$. For each $\varepsilon>0$ we can choose $k$ large enough so that all the intervals $J_0,\hdots,J_{N_k}$ of $F^{(k)}_x$ have length $\le \varepsilon$. By construction, for every $y\in \Cantor$ that extends $x[k]$ we have $F_y\cap J_i \neq \emptyset$ (i.e.\ none of the intervals is ever removed completely) and $F_y\subset J_0 \cup \hdots \cup J_{N_k}$ (i.e.\ nothing is ever added outside of $F^{(k)}_x$). This implies that
	\[ \hmetric(F_x, F_y) \le \max\{ \diam{J_i}\st i\le N_k \} \le \varepsilon, \]
	which proves the continuity.
	
	If $x\in Q_2$ then $x$ is eventually null (i.e.\ there are finitely many $1$s in $x$). Letting $s$ be the largest index s.t.\ $x(s)=1$ (or $s=0$ if there is none) then $F_x= \bigcup_{i\le N_s} R(\alpha, J_i)$. Each set $R(\alpha, J_i)$ is a Salem set of dimension $p$ (as we fixed $\alpha$ accordingly). Since the intervals $J_i$ are closed and disjoint, using \thref{thm:fourier_sup_countable}, we can conclude that $F_x$ is a Salem set of dimension $p$.
	
	On the other hand, if $x\notin Q_2$ then we want to show that $\hdim(F_x)=0$. We will show that for each $s>0$ and $\varepsilon>0$ there is a cover $\sequence{A_n}{n\in\mathbb{N}}$ of $F_x$ s.t.\ $\sum_{n\in\mathbb{N}} \diam{A_n}^s \le  \varepsilon$, i.e.\  for each $s>0$, $\hmeas^s(F_x)=0$. 
	
	For fixed $s$ and $\varepsilon$ we can pick $k$ large enough s.t.\ $2^{-k}\le s$, $2^{-k}\le \varepsilon$ and $x(k+1)=1$. Notice that the intervals $\sequence{H_i}{i\le N_k}$ (as defined in the construction of $F_x$) form a cover of $F_x$ s.t.\ 
	\[ \sum_{i\le N_k} \diam{H_i}^s \le \sum_{i\le N_k} \diam{H_i}^{2^{-k}}\le 2^{-k}\le \varepsilon, \]
	as desired.
\end{proof}	

\begin{proposition}
	\thlabel{thm:dim>p_sigma02_complete}
	For every $p<1$ the sets 
	\begin{align*}
		& \{ A \in\hypCompact([0,1]) \st \hdim(A)> p\},\\
		& \{ A \in\hypCompact([0,1]) \st \fourierdim(A)> p\}
	\end{align*}
	are $\boldfaceSigma^0_2$-complete.
\end{proposition}
\begin{proof}
	The hardness is a straightforward corollary of \thref{thm:sigma02_map_salem}: fix $q$ s.t.\ $p<q<1$ and the $\boldfaceSigma^0_2$-complete subset $Q_2$ of $\Cantor$. We can consider the map $f_q\colon\Cantor \to \closedSalem([0,1])$ as in \thref{thm:sigma02_map_salem} and notice that 
	\[ \dim(f_q(x))>p \iff x\in Q_2. \]
	The completeness follows from \thref{thm:capdim_complexity} and \thref{thm:fourierdim_complexity}.
\end{proof}

Recall that $P_3$ is the $\boldfacePi^0_3$-complete subset of $2^{\mathbb{N}\times\mathbb{N}}$ defined as 
\[ P_3 := \{ x \in 2^{\mathbb{N}\times\mathbb{N}} \st (\forall m)(\forall^\infty n)( x(m,n) = 0 ) \}. \]

\begin{theorem}
	\thlabel{thm:dim_>=p_pi03_complete}
	For every $p\in (0,1]$ there exists a continuous function $F\colon 2^{\mathbb{N}\times\mathbb{N}}\to \hypCompact([0,1])$ s.t.\ for every $x\in 2^{\mathbb{N}\times\mathbb{N}}$, $F(x)$ is a Salem set and $\dim(F(x))\ge p$ iff $x\in P_3$. Letting 
	\begin{align*}
		& X_1:= \{ A \in\hypCompact([0,1]) \st \hdim(A)\ge p\},\\
		& X_2:= \{ A \in\hypCompact([0,1]) \st \fourierdim(A)\ge p\}
	\end{align*}
	we have that every set $X$ s.t.\ $X_2\subset X \subset X_1$ is $\boldfacePi^0_3$-hard. In particular, $X_1$ and $X_2$ are $\boldfacePi^0_3$-complete.
\end{theorem}
\begin{proof}
	The last statement follows from the first one using \thref{thm:capdim_complexity} and \thref{thm:fourierdim_complexity}.
	Consider the (continuous) map $\Phi\colon 2^{\mathbb{N}\times\mathbb{N}}\to 2^{\mathbb{N}\times\mathbb{N}}$ defined as
	\[ \Phi(x)(m,n):= \max_{i\le m} \, x(i,n).\]
	Notice that $P_3=\Phi^{-1}(P_3)$ and that 
	\[ x\notin P_3 \iff (\exists k)(\forall m\ge k)(\exists^\infty n) (\Phi(x)(m,n)=1). \]
	Intuitively, we are modifying the $\mathbb{N}\times\mathbb{N}$ matrix $x$ so that if there is a row of $x$ that contains infinitely many $1$s, then, from that row on, every row will contain infinitely many $1$s.
	% This implies that, if we let $X:=\ran(\Phi)$, the set $Q_3:=P_3\cap X$ is a $\boldfacePi^0_3$-complete subset of $X$. 

%	To show that $X_1$ and $X_2$ are $\boldfacePi^0_3$-hard, 
	We build a continuous map $f\colon 2^{\mathbb{N}\times\mathbb{N}} \to\hypCompact([0,1])$ s.t.\ $F:=f\circ \Phi$ is the desired function.
	
	For every $n$, let $T_n:=[2^{-n-1},2^{-n}]$, $q_{n}:=p(1-2^{-n-1})$ and
	consider the function $f_{q_n}\colon \Cantor \to \closedSalem([0,1])$ of \thref{thm:sigma02_map_salem}. 
	Fix also a similarity transformation $\tau_n \colon [0,1] \to T_n$ and define $g_n\colon \Cantor \to \closedSalem(T_n)$ as $g_n:= \tau_n f_{q_n}$, so that, by \thref{thm:fourier_invariant_affine},
	\[ \dim(g_n(y))= \begin{cases}
		q_n & \text{if } y \in Q_2,\\
		0 	& \text{if } y \notin Q_2.
	\end{cases} \]
	% \begin{itemize}
	% 	\item if $y\in Q_2$ then $g_n(y)$ is a Salem set of dimension $q_n$;
	% 	\item if $y\notin Q_2$ then $g_n(y)$ is a Salem set of dimension $0$.
	% \end{itemize}
	Let $x_m$ be the $m$-th row of $x\in 2^{\mathbb{N}\times\mathbb{N}}$. We define 
	\[ f(x):= \{0\} \cup \bigcup_{m\in\mathbb{N}} g_m(x_m). \]
	Intuitively, we are dividing the interval $[0,1]$ into countably many intervals and, on each interval, we are applying the construction we described in the proof of \thref{thm:sigma02_map_salem} (proportionally scaled down). The continuity of $f$ follows from the continuity of each $g_m$. The accumulation point $0$ is added to ensure that $f(x)$ is a closed set.
	
	Recall that Hausdorff dimension is stable under countable unions, so 
	\[ \hdim(f(x))= \sup_{m\in\mathbb{N}} \hdim( g_m(x_m) ).\]
	
	Moreover, since the sets $\family{T_m}{m\in\mathbb{N}}$ are closed,	we can apply \thref{thm:fourier_sup_countable} and conclude that 
	\[ \fourierdim(f(x))= \sup_{m\in\mathbb{N}} \fourierdim( g_m (x_m) ). \]
	Since each $g_m(x_m)$ is Salem we have that $f(x)$ is Salem (and, in turn, $F(x)$ is Salem) and 
	\[ \dim(f(x))= \sup_{m\in\mathbb{N}} \hdim( g_m(x_m) ) = \sup_{m\in\mathbb{N}} \fourierdim( g_m(x_m) ).\]

	If $x\in P_3$ then $\Phi(x)\in P_3$ and, for every $m$, $\Phi(x)_m\in Q_2$. This implies that $g_m(\Phi(x)_m)$ is a Salem set of dimension $q_m$ and therefore 
	\[ \dim(F(x))= \sup_{m\in\mathbb{N}} q_m  = p. \]
	On the other hand, if $x\notin P_3$ then there is a $k>0$ s.t.\ for every $m\ge k$, $\Phi(x)_m\notin Q_2$ and hence $\dim(g_m(\Phi(x)_m))=0$. This implies that 
	%\[ \hdim(f(x)) = \sup_{m\in\mathbb{N}} \hdim( g_m(x_m) ) \ge \frac{1}{2} \]
	\[ \dim(F(x))  \le q_k < p, \]
	and this completes the proof.
\end{proof}

\begin{theorem}
	\thlabel{thm:salem_pi03_complete}
	The set $\{ A\in\hypCompact([0,1]) \st A\in \Salem([0,1])\}$ is $\boldfacePi^0_3$-complete.
\end{theorem}
\begin{proof}
	Let $K\in \hypCompact([0,1])$ be s.t.\ $\hdim(K)=p$ and $\fourierdim(K)=0$. Let also $F$ be the map provided by \thref{thm:dim_>=p_pi03_complete} and define the map $h\colon 2^{\mathbb{N}\times\mathbb{N}}\to\hypCompact([0,1])$ as 
	\[ h(x):= F(x)\cup K.  \]
	Then $h$ is continuous (see e.g.\ \cite[Ex.\ 4.29(iv)]{KechrisCDST}) and
	\begin{gather*}
		\hdim(h(x)) = \max \{ \dim(F(x)), p \},\\
		\fourierdim(h(x)) = \dim(F(x)).
	\end{gather*}
	In particular, $h(x)$ is Salem iff $\dim(F(x))\ge p$ iff $x\in P_3$. The claim follows by \thref{thm:salem_pi03}.
\end{proof}
	
This shows that the upper bounds we obtained in \thref{thm:capdim_complexity}, \thref{thm:fourierdim_complexity} and \thref{thm:salem_pi03} are sharp. In particular, since $\hypCompact([0,1])$ is a Polish space, this implies that the hyperspace of closed Salem subsets of $[0,1]$ is not a Polish space (in the relative topology). This follows from \cite[Thm.\ 3.11]{KechrisCDST}, as a subset of a Polish space is Polish iff it is $\boldsymbol{G}_\delta$.

Notice that, if we endow $\closedSalem([0,1])$ with the upper topology induced by $\hypCompactUF([0,1])$ then, by \thref{thm:capdim_complexity} (or, equivalently, by \thref{thm:fourierdim_complexity}), we have that 
\begin{gather*}
	\{ (A,p) \in\closedSalem([0,1])\times [0,1] \st \dim(A)> p\}\text{ is }\boldfaceSigma^0_2,\\
	\{ (A,p) \in\closedSalem([0,1])\times [0,1] \st \dim(A)\ge p\}\text{ is } \boldfacePi^0_3.
\end{gather*}
Moreover, the proofs of \thref{thm:dim>p_sigma02_complete} and \thref{thm:dim_>=p_pi03_complete} show that, for every $p<1$ and every $q>0$,
\begin{gather*}
	Q_2 \wadgereducible \{ A \in \closedSalem([0,1]) \st \dim(A) > p \},\\
	P_3 \wadgereducible \{ A \in \closedSalem([0,1]) \st \dim(A) \ge q \}.
\end{gather*} 
However we cannot say that they are complete for their respective classes, because the definition of completeness requires the ambient space to be Polish, and $\closedSalem([0,1])$ is not.

Recall that the Fourier dimension of $A$ is based on an estimate on the decay of the Fourier transform of a probability measure supported on $A$. In particular $\fourierdim(A)=\sup \{ \fourierdim(\mu)\st \mu \in \ProbabilityMeas(A) \}$. This is equivalent to let $\mu$ range over finite (non-trivial) Radon measures on $A$, as the estimate on the decay of the Fourier transform is only up to a multiplicative constant. One may wonder whether it is possible to strengthen this condition by defining the Fourier dimension of $A$ as 
\[ \sup\{ s\in [0,1]\st (\exists \mu\in\ProbabilityMeas(A))(\forall x \in \mathbb{R})(|\fouriertransform{\mu}(x)|\le |x|^{-s/2} )\}. \]
The $\boldfacePi^0_3$-completeness of $\closedSalem([0,1])$ implies that the notion of dimension we would obtain is different. Indeed, the space $\ProbabilityMeas([0,1])$ is a compact space, while the space $[0,\infty)\times \ProbabilityMeas([0,1])$ is not. In particular, removing the constant $c$ in the condition on the decay of the Fourier transform would imply that $\closedSalem([0,1])$ is $\boldfacePi^0_2$ (as the projection of a closed set along a compact space is closed, see the proofs of \thref{thm:capdim_complexity} and \thref{thm:fourierdim_complexity}), and therefore not $\boldfacePi^0_3$-complete.

\section{The complexity of closed Salem subsets of \texorpdfstring{$[0,1]^d$}{[0,1]d}}
\label{sec:salem_[0,1]d}
Let us now turn our attention to the family of closed Salem subsets of $[0,1]^d$. 

\begin{proposition}
	\thlabel{thm:n_dim_complexity}
	For every $d\ge 1$:
	\begin{enumerate}
		\item $\{ (A,p) \in\hypCompactUF([0,1]^d)\times [0,d] \st \hdim(A)> p\}$ is $\boldfaceSigma^0_2$;
		\item $\{ (A,p) \in\hypCompactUF([0,1]^d)\times [0,d] \st \hdim(A)\ge p\}$ is $\boldfacePi^0_3$;
		\item $\{ (A,p) \in\hypCompactUF([0,1]^d)\times [0,d] \st \fourierdim(A)> p\}$ is $\boldfaceSigma^0_2$;
		\item $\{ (A,p) \in\hypCompactUF([0,1]^d)\times [0,d] \st \fourierdim(A)\ge p\}$ is $\boldfacePi^0_3$;
		\item $\{ A \in\hypCompactUF([0,1]^d)\st A \in \Salem([0,1]^d)\}$ is $\boldfacePi^0_3$.
	\end{enumerate}
\end{proposition}
\begin{proof}
	For the first two points, the proof is a straightforward adaptation of the proof of \thref{thm:capdim_complexity}. Indeed, recall that Frostman's lemma holds for Borel subsets of $\mathbb{R}^d$, hence we can characterize the Hausdorff dimension by means of the capacitary dimension. Moreover, since $[0,1]^d$ is compact, the condition $\mu(B(x,r))\le c r^s$ is closed and the space $\ProbabilityMeas([0,1]^d)$ is compact. Therefore $\capdim(A)$ is the supremum of a $\boldfaceSigma^0_2$ set, from which the claim follows.

	Similarly, points $3$ and $4$ follow by adapting the proof of \thref{thm:fourierdim_complexity}. Indeed the map $F:=(\mu,x)\mapsto |\fouriertransform{\mu}(x)|$ is continuous and the condition $|\fouriertransform{\mu}(x)|>c |x|^{-t/2}$ is open, therefore the Fourier dimension is the supremum of a $\boldfaceSigma^0_2$ set.

	Finally, the last point can be proved by following the proof of \thref{thm:salem_pi03} and using points $1$ and $3$.
\end{proof}

The fact that the lower bounds for the complexity of the above sets are tight does not come as a corollary of the results in the $1$-dimensional case. Indeed, it is well known that the Fourier dimension is sensitive to the ambient space: any $m$-dimensional hyperplane has null Fourier dimension when seen as a subset of $\mathbb{R}^d$, with $d>m$ (in particular, the unit interval $[0,1]$ has full Fourier dimension if seen as a subset of itself or of $\mathbb{R}$, but it has null Fourier dimension if seen as a subset of $\mathbb{R}^2$).

A natural approach to lifting the lower bounds from the $1$-dimensional case to higher dimensions is to exploit the theorem of Gatesoupe \cite{Gatesoupe67} mentioned in the introduction: if $A\subset [0,1]$ has at least two points and is Salem with dimension $\alpha$ then the radial set $r(A) := \{ x\in [-1,1]^d \st |x|\in A\}$ is a Salem set with dimension $d-1+\alpha$. Thus, given $p>d-1$, consider the function $h_p \colon \Cantor \to \hypCompact([-1,1]^d)$ defined as $h_p(x):= r(f_{p-d+1}(x)\cup \{1\})$, where $f_{p-d+1}$ is the function from \thref{thm:sigma02_map_salem}. It is immediate that $h_p$ is continuous and, by the above observation, we have that $h_p(x)$ is Salem for every $x$ and 
\[ \dim(h_p(x)) = \begin{cases}
	p & \text{if } x\in Q_2\\
	d-1 & \text{if } x\notin Q_2.
\end{cases}\]
Using $h_p$ we can mimic the proofs of Section~\ref{sec:salem_[0,1]} to obtain the analogue of \thref{thm:dim>p_sigma02_complete} for $p\ge d-1$, the analogue of \thref{thm:dim_>=p_pi03_complete} when $p>d-1$, and the analogue of \thref{thm:salem_pi03_complete}. 

Of course, the downside of this approach is that it does not deal with the case $p< d-1$. To fully generalize the results of Section~\ref{sec:salem_[0,1]}, we prove a $d$-dimensional analogue of \thref{thm:sigma02_map_salem}. In recent work, Fraser and Hambrook \cite{FraserHambrook2019} presented a construction of a Salem subset of $[0,1]^d$ of dimension $p$, for every $p\in [0,d]$.  

\begin{definition}[{\cite{FraserHambrook2019}}]
	\thlabel{def:ndim_jarnik}
	Let $K$ be a number field of degree $d$, i.e.\ $K$ is a field extension of $\mathbb{Q}$ and $\dim_\mathbb{Q} K = d$. Let $B=\{\omega_0,\hdots,\omega_{d-1} \}$ be an integral basis for $K$. We can identify $\mathbb{Q}^d$ with $K$ by mapping a vector $q=(q_0,\hdots, q_{n-1})$ to $\sum_{i<n} q_i \omega_i \in K$. Moreover, since $B$ is an integral basis, we can also identify $\mathbb{Z}^d$ with the ring of integers $\intring{K}$ for $K$. For every $\alpha\ge 0$, we define the set $E(K,B,\alpha)$ as 
	\[ \left\{ x\in [0,1]^d \st (\exists^\infty (q,r)\in \mathbb{Z}^d\times \mathbb{Z}^d) \left( \left\|x- \frac{r}{q}\right\|_\infty \le \| q \|_\infty^{-2-\alpha} \right)\right\},  \]
	where $\|\cdot \|_\infty$ is the max-norm on $\mathbb{R}^d$.
\end{definition}

The set $E(K,B,\alpha)$ is a higher dimensional analogue of the fractal $E(\alpha)$. 

\begin{theorem}[{\cite[Thm.\ 4.1]{FraserHambrook2019}}]
	\thlabel{thm:FH_fractal_is_salem}
	For every $\alpha\ge 0$, the set $E(K,B,\alpha)$ is a Salem set of dimension $2d/(2+\alpha)$.
\end{theorem}

The fact that $E(K,B,0)$ is Salem of dimension $d$ is not explicitly mentioned in \cite{FraserHambrook2019}, but a simple proof was suggested by Hambrook (personal communication): indeed it is enough to notice that, for every $\alpha$ and every $\varepsilon>0$, $E(K,B,\alpha + \varepsilon) \subset E(K,B,\alpha)$, and therefore the claim follows from the monotonicity of the Fourier dimension.

Notice that, in general, the set $E(K,B,\alpha)$ is not closed but $\boldfacePi^0_3$. Analogously to the one-dimensional case, the proof of \thref{thm:FH_fractal_is_salem} shows that there is a closed Salem subset $S(K,B,\alpha)$ of $E(K,B,\alpha)$ with dimension $2d/(2+\alpha)$. To prove the following \thref{thm:sigma02_map_salem_ndim} we cannot proceed as in the one-dimensional case, as we do not know whether $E(K,B,\alpha)$ is dense in $[0,1]^d$.

\begin{lemma}
	\thlabel{thm:sigma02_map_salem_ndim}
	Fix $d>0$. For every $p\in [0,d]$ there exists a continuous map $f_p\colon \Cantor\to \hypCompact([0,1]^d)$ s.t.\ for every $x$, $f_p(x)$ is Salem and 
	\[ \dim(f_p(x)) = \begin{cases}
		p & \text{if } x\in Q_2\\
		0 & \text{if } x\notin Q_2
	\end{cases}\]
\end{lemma}
\begin{proof}
	The idea of the proof is similar to the one of \thref{thm:sigma02_map_salem}: given $x\in\Cantor$, we define a closed set $F_x$ by following the construction of the set $S(K,B,\alpha)$, having care of controlling the Hausdorff dimension whenever $x(k)=1$.
	
	Formally, let $p>0$ (otherwise the claim follows trivially by considering the map $x\mapsto \emptyset$) and let $\alpha$ s.t.\ $2d/(2+\alpha)=p$. 

	Fix $K$ and $B$ as in \thref{def:ndim_jarnik}. For the sake of readability, let $S(\alpha):=S(K,B,\alpha)$. We can write $S(\alpha)$ as intersection of closed nested sets $S^{(k)}(\alpha)$ defined as 
	\[ S^{(k)}(\alpha):= \{ y\in [0,1]^d \st d(y,S(\alpha))\le 2^{-k} \}. \]
	Clearly, $S^{(k)}(\alpha)$ is closed with non-empty interior.

	For each non-degenerate hypercube $C$, define $S(\alpha, C):=\tau(S(\alpha)) $, where $\tau$ is a similarity transformation that maps $[0,1]^d$ onto $C$, and $S^{(k)}(\alpha,C)$ accordingly.

	We define $F^{(k)}_x$ recursively, ensuring that, for each $k$, $F^{(k)}_x$ is closed and has non-empty interior, and $F^{(k+1)}_x\subset F^{(k)}_x$:
	\begin{description}
		\item[Stage $k=0$] $F^{(0)}_x := C_0 := [0,1]^d$, $P_0:=\emptyset$;
		\item[Stage $k+1$] If $x(k+1)=1$, let $P_k:= \{ p^{(k)}_i \}_{i\le N_k}$ be a finite set of points in $F^{(k)}_x$ s.t.\ for each $t\in F^{(k)}_x$ there exists $i\le N_k$ s.t.\ $|t-p^{(k)}_i|\le 2^{-(k+1)}$. Let $C_k$ be the largest (non-degenerate) hypercube contained in $F^{(k)}_x$. 
		Define 
		\[ F^{(k+1)}_x:= S^{(0)}(\alpha,C_k) \cup P_k.\]
		If $x(k+1)=0$ then let $s< k$ be largest s.t.\ $x(s+1)=1$ (or $s=0$ if there is none). We define 
		\[ F^{(k+1)}_x:= S^{(k+1-s)}(\alpha,C_s) \cup P_s. \]
	\end{description}  
	Define $f_p:=x\mapsto F_x = \bigcap_{k\in\mathbb{N}} F^{(k)}_x$. Clearly $F_x$ is closed, as intersection of closed sets. The continuity of the map $f_p$ is guaranteed by the fact that for each $k$, 
	\[ \hmetric(F^{(k)}_x, F^{(k+1)}_x)\le 2^{-(k+1)}.\] 
	This follows from our choice of $P_k$ in the first case, and the fact that $\hmetric(S^{(k)}(\alpha), S^{(k+1)}(\alpha))\le 2^{-(k+1)}$ in the second case.

	Adapting the proof of \thref{thm:sigma02_map_salem}, it is possible to show that $F_x$ is Salem and that $\dim(F_x)=p$ iff $x\in Q_2$. 
\end{proof}

From \thref{thm:sigma02_map_salem_ndim} we can derive the following results, as we did with their analogues in the previous section. 

\begin{proposition}
	\thlabel{thm:dim>p_sigma02_complete_ndim}
	For every $p\in [0,d)$ the sets 
	\begin{align*}
		& \{ A \in\hypCompact([0,1]^d) \st \hdim(A)> p\},\\
		& \{ A \in\hypCompact([0,1]^d) \st \fourierdim(A)> p\}
	\end{align*}
	are $\boldfaceSigma^0_2$-complete.
\end{proposition}

\begin{theorem}
	\thlabel{thm:dim_>=p_pi03_complete_ndim}
	For every $p\in (0,d]$ there exists a continuous function $F\colon 2^{\mathbb{N}\times\mathbb{N}}\to \hypCompact([0,1]^d)$ s.t.\ for every $x\in 2^{\mathbb{N}\times\mathbb{N}}$, $F(x)$ is a Salem set and $\dim(F(x))\ge p$ iff $x\in P_3$. Letting 
	\begin{align*}
		& X_1:= \{ A \in\hypCompact([0,1]^d) \st \hdim(A)\ge p\},\\
		& X_2:= \{ A \in\hypCompact([0,1]^d) \st \fourierdim(A)\ge p\}
	\end{align*}
	we have that every set $X$ s.t.\ $X_2\subset X \subset X_1$ is $\boldfacePi^0_3$-hard. In particular, $X_1$ and $X_2$ are $\boldfacePi^0_3$-complete.
\end{theorem}

\begin{theorem}
	\thlabel{thm:salem_pi03_complete_ndim}
	The set $\{ A\in\hypCompact([0,1]^d) \st A\in \Salem([0,1]^d)\}$ is $\boldfacePi^0_3$-complete.
\end{theorem}

\section{The complexity of closed Salem subsets of \texorpdfstring{$\mathbb{R}^d$}{Rd}}
\label{sec:R^d}
Let us now turn our attention to the closed Salem subsets of $\mathbb{R}^d$. % As explained in Section~\ref{sec:backgroud}, the Vietoris topology on $\hypClosed(X)$ fails to be metrizable, and hence Polish, if the ambient space is not compact. 
In this section, we determine the descriptive complexity of the conditions $\hdim(A)> p$, $\hdim(A)\ge p$, $\fourierdim(A)>p$, $\fourierdim(A)\ge p$, $A\in\Salem(\mathbb{R}^d)$, when $A$ is a closed subset of $\mathbb{R}^d$ and $p\in \mathbb{R}$.

Unless otherwise mentioned, $\hypClosed(X)$ will be endowed with the Fell topology. On the other hand, we will write $\hypClosedUF(X)$ (resp.\ $\hypClosedV(X)$) for the hyperspace of closed subsets of $X$ endowed with the upper Fell topology (resp.\ Vietoris topology). 

The hardness results lift easily from the compact cases. 

\begin{proposition}
	\thlabel{thm:ndim_hardness}
	For every $p\in (0,d]$ and every $q\in [0,d)$, we have  
	\begin{itemize}
		\item $\{ A \in\hypClosed(\mathbb{R}^d) \st \hdim(A)> q\}$ is $\boldfaceSigma^0_2$-hard;
		\item $\{ A \in\hypClosed(\mathbb{R}^d) \st \hdim(A)\ge  p\}$ is $\boldfacePi^0_3$-hard;
		\item $\{ A \in\hypClosed(\mathbb{R}^d) \st \fourierdim(A)>  q\}$ is $\boldfaceSigma^0_2$-hard;
		\item $\{ A \in\hypClosed(\mathbb{R}^d) \st \fourierdim(A)\ge  p\}$ is $\boldfacePi^0_3$-hard;
		\item $\{ A \in\hypClosed(\mathbb{R}^d) \st A \in \Salem(\mathbb{R}^d)\}$ is $\boldfacePi^0_3$-hard.
	\end{itemize}
\end{proposition}
\begin{proof}
	This is a corollary of \thref{thm:dim>p_sigma02_complete_ndim}, \thref{thm:dim_>=p_pi03_complete_ndim} and \thref{thm:salem_pi03_complete_ndim}. Indeed, since the Fourier and Hausdorff dimensions of $A\subset [0,1]^d$ do not change if we see $A$ as a subset of $\mathbb{R}^d$, it is enough to notice that the inclusion map $\hypCompact([0,1]^d)\hookrightarrow\hypClosed(\mathbb{R}^d)$ is continuous. 
\end{proof}

Notice that, since the inclusion $\hypCompact([0,1]^d)\hookrightarrow\hypClosedV(\mathbb{R}^d)$ is continuous as well, the same proof provides a lower bound for the above conditions when the hyperspace $\hypClosed(\mathbb{R}^d)$ is endowed with the Vietoris topology\footnote{However, since $\hypClosedV(\mathbb{R}^d)$ is not Polish, we cannot say that the conditions are hard for their respective classes.}.

As in the previous sections, we obtain the upper bounds for $\hypClosedUF(\mathbb{R}^d)$. % endowing $\hypClosed(\mathbb{R}^d)$ with the upper Fell topology (i.e.\ working in $\hypClosedUF(\mathbb{R}^d)$). 
This will yield, as a corollary, that each of the above set is complete in $\hypClosed(\mathbb{R}^d)$ for its respective class (in case of the upper Fell or Vietoris topology we only obtain a Wadge-equivalence).

Since the proofs of \thref{thm:capdim_complexity}, \thref{thm:fourierdim_complexity} and \thref{thm:n_dim_complexity} exploit the compactness of the ambient space, some extra care is needed when working in a non-compact environment.

\begin{lemma}
	\thlabel{thm:compact_complexity}
	${}$
	\begin{itemize}
		\item $\{ (K,p) \in\hypCompactUF(\mathbb{R}^d)\times [0,d] \st \hdim(K)> p\}$ is $\boldfaceSigma^0_2$;
		\item $\{ (K,p) \in\hypCompactUF(\mathbb{R}^d)\times [0,d] \st \hdim(K)\ge p\}$ is $\boldfacePi^0_3$.
	\end{itemize}
\end{lemma}
\begin{proof}
	Define the set $D(K)$ as 
	\[ \{ s\in [0,d] \st (\exists \mu \in \ProbabilityMeas(K))(\exists c>0)(\forall x\in \mathbb{R}^d)(\forall r>0)(\mu(\ball{x}{r})\le c r^s ) \} \]	
	and recall that $\hdim(K) = \sup D(K)$. For every $n$, let $K_n:=\closure{\ball{\mathbf{0}}{n}}$. Observe that 
	\[ \mu\in\ProbabilityMeas(K) \iff \mu \in \ProbabilityMeas(\mathbb{R}^d) \land \mu(K)\ge 1 \land (\exists n\in\mathbb{N})(\mu(K_n)\ge 1).  \]
	We can therefore rewrite $D(K)$ as follows
	\begin{align*}
		D(K)= \{ s\in [0,d] \st  (\exists \mu \in & \ProbabilityMeas(\mathbb{R}^d))(\exists c>0)(\exists n\in\mathbb{N}) \\
			& ( \mu(K)\ge 1 \land \mu(K_n)\ge 1 \land{}\\
			& (\forall x\in \mathbb{R}^d)(\forall r>0)(\mu(\ball{x}{r})\le c r^s )) \}. 
	\end{align*}
	In particular $\mu(K_n)\ge 1$ implies that $\support{\mu}\subset K_n$, hence 
	\[ \mu(\ball{x}{r})\le c r^s \iff \mu(H)\ge 1-c r^s, \]
	where $H:=\closure{\ball{\mathbf{0}}{n+x+r}}\setdifference \ball{x}{r}$. It is routine to prove that the function $\varphi\colon\mathbb{N}\times\mathbb{R}_+\times\mathbb{R}^d\to \hypCompactUF(\mathbb{R}^d)$ that sends $(n,r,x)$ to the above-defined $H$ is continuous. Notice that if we had set $H=K\setdifference \ball{x}{r}$ then the resulting map would not be continuous. This motivates the use of $K_n$ in the above characterization of $D(K)$.
	
%	By \cite[Ex.\ 17.29, p.\ 114]{KechrisCDST}, for every separable metric space $X$ the set
	By \thref{thm:mu(K)>=x_closed} the set 
	\[ \{ (\mu,K,a)\in \ProbabilityMeas(\mathbb{R}^d) \times \hypCompactUF(\mathbb{R}^d)\times \mathbb{R} \st \mu(K)\ge a\} \]
	is closed. In particular the condition $\mu(\ball{x}{r})\le c r^s$ is closed and the set 
	\begin{align*}
		Q:= \{ & (s,\mu)\in [0,d]\times \ProbabilityMeas(\mathbb{R}^d) \st (\exists c>0)(\exists n\in\mathbb{N}) \\
			& ( \mu(K)\ge 1 \land \mu(K_n)\ge 1 \land (\forall x\in \mathbb{R}^d)(\forall r>0)(\mu(\ball{x}{r})\le c r^s )) \}
	\end{align*}
	is $\boldfaceSigma^0_2$. 
	
	Notice that we can equivalently see $Q$ as a subset of $[0,d]\times \bigcup_{n\in\mathbb{N}}\ProbabilityMeas(K_n)$. In particular, $D(K)$ is the projection of a $\boldfaceSigma^0_2$ set along a metrizable and $\mathbf{K}_\sigma$ space (as $\ProbabilityMeas(X)$ is compact if $X$ is). Therefore, using \thref{thm:AM_Sigma02} we can conclude that $D(K)$ is $\boldfaceSigma^0_2$ and that the conditions
	\[ \hdim(K) > p \iff (\exists s\in \mathbb{Q})(s>p \land s\in D(K)),\]
	\[ \hdim(K) \ge p \iff (\forall s\in \mathbb{Q})(s<p \rightarrow s\in D(K))\]
	are $\boldfaceSigma^0_2$ and $\boldfacePi^0_3$ respectively.
\end{proof}

\begin{lemma}
	\thlabel{thm:AsubsetB_closed}
	The set $\{ (A,B)\in \hypClosedUF(\mathbb{R}^d)\times \hypClosed(\mathbb{R}^d) \st B\subset A \}$ is $\boldfacePi^0_1$.
%	The set $\{ (A,B)\in \hypClosed(\mathbb{R}^d)\times \hypClosed(\mathbb{R}^d) \st B\subset A \}$ is $\boldfacePi^0_1$, when $\hypClosed(\mathbb{R}^d)$ is endowed with the Fell (and hence with the Vietoris) topology.
\end{lemma}
\begin{proof}
	It suffices to show that the complement of the set is open. If $B\not\subset A$, fix $x\in B\setdifference A$ and let $\varepsilon:=d(x,A)>0$. Let $\mathcal{U}_1:=\{ F \in \hypClosed(\mathbb{R}^d) \st F \cap \closure{\ball{x}{\varepsilon/2}}= \emptyset\}$ and $\mathcal{U}_2:=\{ F \in \hypClosed(\mathbb{R}^d) \st F\cap \ball{x}{\varepsilon/2} \neq\emptyset \}$. Clearly $\mathcal{U}_1\times \mathcal{U}_2$ is open in $\hypClosedUF(\mathbb{R}^d)\times \hypClosed(\mathbb{R}^d)$, $(A,B)\in \mathcal{U}_1\times \mathcal{U}_2$ and every $(A',B')\in \mathcal{U}_1\times \mathcal{U}_2$ is s.t.\ $B'\not\subset A'$.
\end{proof}

\begin{theorem}
	\thlabel{thm:ndim_hausdorff_complexity}
%	Assume $\hypClosed(\mathbb{R}^d)$ is endowed with either the Fell or the Vietoris topology. 
	The sets 
	\begin{align*}
		X_1 := &\{ (A,p) \in\hypClosedUF(\mathbb{R}^d)\times [0,d] \st \hdim(A)> p\}, \\
		X_2 := &\{ (A,p) \in\hypClosedUF(\mathbb{R}^d)\times [0,d] \st \hdim(A)\ge p\} 
	\end{align*}
	are $\boldfaceSigma^0_2$ and $\boldfacePi^0_3$ respectively. In particular, for every $p\in[0,d)$ and $q\in (0,d]$, the sets  	
	\begin{align*}
		& \{ A \in\hypClosed(\mathbb{R}^d) \st \hdim(A)> p\},\\
		& \{ A \in\hypClosed(\mathbb{R}^d) \st \hdim(A)\ge q\}
	\end{align*}
	are $\boldfaceSigma^0_2$-complete and $\boldfacePi^0_3$-complete respectively. 
\end{theorem}
\begin{proof}
	Notice that, as a consequence of the countable stability of the Hausdorff dimension, we have
	\[ \hdim(A)= \sup \{ \hdim(K) \st K\subset A \text{ and }K\text{ is compact} \},\]
	and therefore
	\[ \hdim(A)>p \iff (\exists K \in \hypClosed(\mathbb{R}^d))(K\subset A \land K\in \hypCompact(\mathbb{R}^d) \land \hdim(K)>p).\]
	Notice that the condition $K\subset A$ is $\boldfacePi^0_1$ by \thref{thm:AsubsetB_closed}. We claim that the conjunction $K\in \hypCompact(\mathbb{R}^d) \land \hdim(K)>p$ is $\boldfaceSigma^0_2$. This follows from the fact $K\in \hypCompact(\mathbb{R}^d)$ is equivalent to $(\exists n)(K\subset \closure {\ball{\mathbf{0}}{n}})$, hence it is $\boldfaceSigma^0_2$ using again \thref{thm:AsubsetB_closed}; moreover, since the inclusion map $\hypClosed(X)\restrict{\hypCompact(X)}\hookrightarrow \hypCompactUF(X)$ is continuous, by \thref{thm:compact_complexity} the condition $\hdim(K)>p$ is $\boldfaceSigma^0_2$ when $K$ is compact.

	This shows that the set $X_1$ is the projection of a $\boldfaceSigma^0_2$ set along $\hypClosed(\mathbb{R}^d)$. Since $\hypClosed(\mathbb{R}^d)$ is compact, we can use \thref{thm:AM_Sigma02} and conclude that $X_1$ is $\boldfaceSigma^0_2$.

	Moreover, since $\hdim(A)\ge p$ iff $(\forall r\in \mathbb{Q})(r< p \rightarrow\hdim(A)>r)$, this also shows that $X_2$ is $\boldfacePi^0_3$. The completeness follows from \thref{thm:ndim_hardness}.
\end{proof}

With a similar strategy, we can characterize the upper bounds for the Fourier dimension:

\begin{theorem}
	\thlabel{thm:ndim_fourier_complexity}
%	Assume $\hypClosed(\mathbb{R}^d)$ is endowed with either the Fell or the Vietoris topology. 
	The sets 
	\begin{align*}
		X_1 := &\{ (A,p) \in\hypClosedUF(\mathbb{R}^d)\times [0,d] \st \fourierdim(A)> p\}, \\
		X_2 := &\{ (A,p) \in\hypClosedUF(\mathbb{R}^d)\times [0,d] \st \fourierdim(A)\ge p\} 
	\end{align*}
	are $\boldfaceSigma^0_2$ and $\boldfacePi^0_3$ respectively. In particular, for every $p\in[0,d)$ and $q\in (0,d]$, the sets  	
	\begin{align*}
		& \{ A \in\hypClosed(\mathbb{R}^d) \st \fourierdim(A)> p\},\\
		& \{ A \in\hypClosed(\mathbb{R}^d) \st \fourierdim(A)\ge q\}
	\end{align*}
	are $\boldfaceSigma^0_2$-complete and $\boldfacePi^0_3$-complete respectively. 
\end{theorem}
\begin{proof}
	Notice that the condition 
	\[(\forall x\in\mathbb{R}^d)(|\fouriertransform{\mu}(x)| \le c|x|^{-s/2} )\]
	is closed, as the map $(\mu,x,s,c)\mapsto |\fouriertransform{\mu}(x)|-c|x|^{-s/2}$ is continuous (see also the proof of \thref{thm:fourierdim_complexity}). In particular, this implies that the sets 
	\begin{gather*}
		\{ (K,p) \in\hypCompactUF(\mathbb{R}^d)\times [0,d] \st \fourierdim(K)> p\},\\
		\{ (K,p) \in\hypCompactUF(\mathbb{R}^d)\times [0,d] \st \fourierdim(K)\ge p\}
	\end{gather*}
	are $\boldfaceSigma^0_2$ and $\boldfacePi^0_3$ respectively.
	
	Since the Fourier dimension is inner regular for compact sets, we can write  
	\[ \fourierdim(A)>p \iff (\exists K \in \hypClosed(\mathbb{R}^d))(K\subset A \land K\in \hypCompact(\mathbb{R}^d) \land \fourierdim(K)>p).\]
	As in the proof of \thref{thm:ndim_hausdorff_complexity}, using \thref{thm:AsubsetB_closed} and the fact that the inclusion map $\hypClosed(X)\restrict{\hypCompact(X)}\hookrightarrow \hypCompactUF(X)$ is continuous, we have that $X_1$ is the projection of a $\boldfaceSigma^0_2$ set along $\hypClosed(\mathbb{R}^d)$. \thref{thm:AM_Sigma02} implies that $X_1$ is $\boldfaceSigma^0_2$ and $X_2$ is $\boldfacePi^0_3$. The completeness follows from \thref{thm:ndim_hardness}.
\end{proof}

\begin{theorem}
	The set $\{ A\in\hypClosed(\mathbb{R}^d) \st A\in \Salem(\mathbb{R}^d)\}$ is $\boldfacePi^0_3$-complete.%, where $\hypClosed(\mathbb{R}^d)$ is endowed with either the Fell or the Vietoris topology.
\end{theorem}
\begin{proof}
	Using \thref{thm:ndim_hausdorff_complexity} and \thref{thm:ndim_fourier_complexity} we have that, for every $p$, the conditions $\hdim(A)>p$ and $\fourierdim(A)>p$ are $\boldfaceSigma^0_2$. The fact that $\{ A\in\hypClosed(\mathbb{R}^d) \st A\in \Salem(\mathbb{R}^d)\}$ is $\boldfacePi^0_3$ follows as in the proof of \thref{thm:salem_pi03}, while the hardness (and hence the completeness) follows from \thref{thm:ndim_hardness}.
\end{proof}

\section{Final remarks}

Let $X$ be $[0,1]^d$ or $\mathbb{R}^d$, for some $d\ge 1$. Notice that the set $\closedSalem(X)$ is comeager in $\hypClosedV(X)$. Indeed, the set $\{K\in\hypClosedV(X)\st \hdim(K)\le 0\}\subset \closedSalem(X)$ is $\boldfacePi^0_2$ by \thref{thm:capdim_complexity} (and its higher-dimensional analogues), and dense because it contains the set $\{ K \in \hypClosedV(X)\st K$ is finite$\}$, which is dense. % (see the proof of \cite[Thm.\ 4.22, p.\ 25]{KechrisCDST}). 
The same argument also shows that for every $p$ the sets $\{K\in\hypClosedV(X)\st \hdim(K)\le p\}$ and $\{K\in\hypClosedV(X)\st \fourierdim(K)\le p\}$ are comeager.

Recall that if $\boldfaceGamma$ is a level in the Borel hierarchy, we say that $f\colon X\to Y$ is $\boldfaceGamma$-measurable if, for every open $U\subset Y$, $f^{-1}(U)\in \boldfaceGamma(X)$. \thref{thm:dim>p_sigma02_complete_ndim}, \thref{thm:ndim_hausdorff_complexity} and \thref{thm:ndim_fourier_complexity} show that the maps $\hdim\colon \hypClosed(X)\to \mathbb{R}$ and $\fourierdim\colon \hypClosed(X)\to \mathbb{R}$ are $\boldfaceSigma^0_3$-measurable. Using \cite[Thm.\ 24.3]{KechrisCDST}, this is equivalent to both $\hdim$ and $\fourierdim$ being Baire class $2$.

\bibliographystyle{mbibstyle}
\bibliography{bibliography}

\providecommand{\bysame}{\leavevmode\hbox to3em{\hrulefill}\thinspace}
\providecommand{\MR}{\relax\ifhmode\unskip\space\fi MR }
\providecommand{\MRhref}[2]{%
  \href{http://www.ams.org/mathscinet-getitem?mr=#1}{#2}
}
\providecommand{\href}[2]{#2}
\begin{thebibliography}{10}

\bibitem{AndrMarcODE97}
Andretta,  Alessandro and Marcone,  Alberto, \emph{Ordinary differential
  equations and descriptive set theory: uniqueness and globality of solutions
  of Cauchy problems in one dimension}, Fundamenta Mathematicae \textbf{153}
  (1997), no.~2, 157--190.

\bibitem{Beer1993}
Beer,  Gerald, \emph{Topologies on Closed and Closed Convex Sets}, 1 ed.,
  Mathematics and Its Applications, vol. 268, Springer, Dordrecht, 1993,
  \doi{10.1007/978-94-015-8149-3}.

\bibitem{Besicovitch1934}
Besicovitch,  A.~S., \emph{Sets of Fractional Dimensions (IV): On Rational
  Approximation to Real Numbers}, Journal of the London Mathematical Society
  \textbf{s1-9} (1934), no.~2, 126--131, \doi{10.1112/jlms/s1-9.2.126}.

\bibitem{Bluhm98}
Bluhm,  Christian, \emph{On a theorem of Kaufman: Cantor-type construction of
  linear fractal Salem sets}, Ark. Mat. \textbf{36} (1998), no.~2, 307--316,
  \doi{10.1007/BF02384771}.

\bibitem{Dagstuhl2016}
Brattka,  Vasco, Kawamura,  Akitoshi, Marcone,  Alberto, and Pauly,  Arno,
  \emph{{Measuring the Complexity of Computational Content (Dagstuhl Seminar
  15392)}}, Dagstuhl Reports \textbf{5} (2016), no.~9, 77--104,
  \doi{10.4230/DagRep.5.9.77}.

\bibitem{EPS2014}
{Ekstr{\"o}m},  Fredrik, {Persson},  Tomas, and {Schmeling},  J{\"o}rg,
  \emph{On the Fourier dimension and a modification}, Journal of Fractal
  Geometry \textbf{2} (2015), no.~3, 309--337, \doi{10.4171/JFG/23}.

\bibitem{ESSurveyFourier}
Ekstr{\"o}m,  Fredrik and Schmeling,  J{\"o}rg, \emph{A Survey on the Fourier
  Dimension}, Patterns of Dynamics (Gurevich,  Pavel, Hell,  Juliette,
  Sandstede,  Bj\"orn, and Scheel,  Arnd, eds.), Springer Proceedings in
  Mathematics \& Statistics, vol. 205, Springer, Cham, 2017,
  \doi{10.1007/978-3-319-64173-7_5}, pp.~67--87.

\bibitem{Falconer14}
Falconer,  Kenneth, \emph{Fractal Geometry, Mathematical Foundations and
  Applications}, 3 ed., John Wiley \& Sons, Ltd., Chichester, 2014.

\bibitem{FraserHambrook2019}
Fraser,  Robert and Hambrook,  Kyle, \emph{Explicit Salem sets in
  $\mathbb{R}^n$}, September 2019, available at
  \url{https://arxiv.org/abs/1909.04581v2}.

\bibitem{Gatesoupe67}
Gatesoupe,  Michel, \emph{Sur un th{\'e}or{\`e}me de R. Salem}, Bulletin des
  Sciences Math{\'e}matiques. Deuxi{\`e}me S{\'e}rie \textbf{91} (1967),
  125--127.

\bibitem{HambrookR2}
Hambrook,  Kyle, \emph{Explicit Salem sets in $\mathbb{R}^2$}, Advances in
  Mathematics \textbf{311} (2017), 634--648, \doi{10.1016/j.aim.2017.03.009}.

\bibitem{Jarnik1928}
Jarn{\'\i}k,  Vojt{\u e}ch, \emph{Zur metrischen Theorie der diophantischen
  Approximationen}, Prace Matematyczno-Fizyczne \textbf{36} (1928-1929), no.~1,
  91--106.

\bibitem{Kahane1970}
Kahane,  Jean-Pierre, \emph{Sur certains ensembles de Salem}, Acta Mathematica
  Academiae Scientiarum Hungarica \textbf{21} (1970), no.~1, 87--89,
  \doi{10.1007/BF02022490}.

\bibitem{Kahane1993}
\bysame, \emph{Some Random Series of Functions}, 2 ed., Cambridge Studies in
  Advanced Mathematics, Cambridge University Press, 1993.

\bibitem{Kaufmann81}
Kaufman,  Robert, \emph{On the theorem of Jarn\'ik and Besicovitch}, Acta
  Arithmetica \textbf{39} (1981), no.~3, 265--267.

\bibitem{KechrisCDST}
Kechris,  Alexander~S., \emph{Classical Descriptive Set Theory}, 1 ed.,
  Springer-Verlag, 1995.

\bibitem{Ke70}
Keesling,  James, \emph{Normality and Properties Related to Compactness in
  Hyperspaces}, Proceedings of the American Mathematical Society \textbf{24}
  (1970), no.~4, 760--766, \doi{10.1090/S0002-9939-1970-0253292-7}.

\bibitem{KleinThom84}
Klein,  Erwin and Thompson,  Anthony~C., \emph{Theory of Correspondences -
  Including Applications to Mathematical Economics}, Wiley, 1984.

\bibitem{Korner11}
K{\"o}rner,  Thomas~William, \emph{Hausdorff and Fourier dimension}, Studia
  Mathematica \textbf{206} (2011), no.~1, 37--50, \doi{10.4064/sm206-1-3}.

\bibitem{Mattila95}
Mattila,  Pertti, \emph{Geometry of Sets and Measures in Euclidean Spaces:
  Fractals and Rectifiability}, Cambridge University Press, 1995,
  \doi{10.1017/CBO9780511623813}.

\bibitem{MattilaFA}
\bysame, \emph{Fourier Analysis and Hausdorff Dimension}, 1 ed., Cambridge
  University Press, 2015, \doi{10.1017/CBO9781316227619}.

\bibitem{Salem1950}
Salem,  R., \emph{On singular monotonic functions whose spectrum has a given
  Hausdorff dimension}, Arkiv f{\"o}r Matematik \textbf{1} (1951), no.~4,
  353--365, \doi{10.1007/BF02591372}.

\bibitem{SteinWeiss}
Stein,  Elias~M. and Weiss,  Guido, \emph{Introduction to Fourier analysis on
  Euclidean spaces}, 1 ed., Princeton University Press, Princeton, N.J., 1971,
  \doi{10.1515/9781400883899}.

\bibitem{Wolff03}
Wolff,  Thomas~H., \emph{Lectures on harmonic analysis}, 1 ed., American
  Mathematical Society, September 2003, \doi{10.1090/ulect/029}.

\end{thebibliography}

\printauthor

\end{document}